\documentclass[a4paper, oneside, 11pt]{article}

\usepackage[utf8]{inputenc} 
\usepackage[T1]{fontenc} 
\usepackage{textcomp} 
\usepackage[english]{babel} 
\usepackage{indentfirst}
\usepackage[autolanguage]{numprint} 
\usepackage{hyperref} 
\hypersetup{colorlinks=true,linkcolor=blue,citecolor=red,urlcolor=blue,pdfborderstyle={/S/U/W > 1}}
\usepackage{graphicx} 
\usepackage{verbatim} 
\usepackage{makeidx} 
\usepackage{enumitem}
\usepackage{tikz}
\usetikzlibrary{matrix,arrows}
\usepackage{etoolbox}
\usepackage{url}

\usepackage{amsmath,amssymb,amsthm} 
\usepackage[all]{xy} 

\renewcommand\epsilon\varespilon 

\newcommand\NN{\mathbb{N}} 
\newcommand\ZZ{\mathbb{Z}} 
\newcommand\RR{\mathbb{R}} 
\newcommand\gts[1]{``#1''} 
\swapnumbers 
\theoremstyle{definition} 
\newtheorem{Def}{Definition}[subsection] 
\newtheorem{Def,Thm}[Def]{Theorem-definition} 
\newtheorem{Def,Prop}[Def]{Proposition-definition} 
\newtheorem{Not}[Def]{Notation} 

\theoremstyle{plain} 
\newtheorem{Prop}[Def]{Proposition} 
\newtheorem{Lem}[Def]{Lemma} 
\newtheorem{Thm}[Def]{Theorem} 
\newtheorem{Cor}[Def]{Corollary} 

\theoremstyle{remark} 
\newtheorem{Ex}[Def]{Example} 
\newtheorem{Exs}[Def]{Examples} 
\newtheorem{Rq}[Def]{Remark} 

\makeindex

\begin{document} 

\title{On~profinite~subgroups of an~algebraic~group over~a~local~field}

\author{B.~Loisel}




\maketitle 

\begin{abstract}
The purpose of this paper is to link anisotropy properties of an algebraic group together with compactness issues in the topological group of its rational points.
We find equivalent conditions on a smooth affine algebraic group scheme over a non-Archimedean local field for the associated rational points to admit maximal compact subgroups.
We use the structure theory of pseudo-reductive groups provided, whatever the characteristic, by Conrad, Gabber and Prasad.
We also investigate thoroughly maximal pro-$p$ subgroups in the semisimple case, using Bruhat-Tits theory.
\end{abstract}

\tableofcontents

\section{Introduction}

Given a base field $k$ and an affine algebraic $k$-group denoted by $G$, we get an abstract group called the group of rational points, denoted by $G(k)$.
By the term {\em algebraic $k$-group}, we mean in this paper that $G$ is a group scheme defined over $k$ of finite type.
Unless stated otherwise, all algebraic $k$-groups will be assumed to be smooth and affine, but not necessarily connected, that is to say a linear algebraic group in the terminology of \cite{Borel} for instance.
When $k$ is a topological field, this group inherits a topology from the field.
It is a natural problem to link some algebraic properties of an algebraic $k$-group $G$ and topological properties of its rational points $G(k)$.
In this article, we consider a non-Archimedean local field $k$, hence the topological group $G(k)$ will be totally disconnected  and locally compact.
Thus, one can investigate the compact, equivalently profinite, subgroups of $G(k)$.
In the following, we denote by $\omega$ the discrete valuation, $\mathcal{O}_k$ the ring of integers, $\mathfrak{m}$ its maximal ideal, $\varpi$ a uniformizer, and $\kappa = \mathcal{O}_k / \mathfrak{m}$ the residue field.
We denote by $\overline{k}$ an algebraic closure of $k$.

\subsection{Existence of maximal compact subgroups}
Starting from the algebraic $k$-group $G$, one is led to consider the group of rational points $G(k)$ endowed with the topology induced by that of the base field $k$.
We would like to get a correspondence between algebraic properties of $G$ and topological properties of $G(k)$.
A theorem of Bruhat and Tits makes a link between anisotropy and compactness \cite[5.1.27]{BruhatTits2} for reductive groups.
Another link between algebra and topology is Godement's compactness criterion for arithmetic quotients of non-Archimedean Lie groups, recently extended to positive characteristic by Conrad \cite[A5]{Conrad-cosetfinite}.
In the first part, we obtain further results for a general algebraic group over a local field; more precisely, we provide a purely algebraic condition on the $k$-group $G$ for $G(k)$ to admit maximal compact subgroups. The fact that this condition is non-trivial is roughly explained by the following:

\begin{Exs}\label{ex:compact:open:subgroups}
Consider the additive group $\mathbb{G}_{a,k}$.
Inside the topological group $(k,+)$, the subgroups $\varpi^n \mathcal{O}_k$, where $n \in \NN$ form a basis of compact open neighbourhoods of the neutral element $0$.
However, $k$ is not compact and does not admit a maximal compact subgroup, since $k$ is the union $\bigcup_{n \in \ZZ} \varpi^n \mathcal{O}_k$ of compact subgroups.
Moreover, $(k,+)$ is not compactly generated.

On the opposite, consider the multiplicative group $\mathbb{G}_{m,k}$.
The topological group $k^\times$ has a unique maximal compact subgroup: $\mathcal{O}_k^\times$.
Since $k$ is assumed to be discretely valued by $\omega : k^\times \rightarrow \ZZ$, the topological group $k^\times$ is compactly generated by $\mathcal{O}_k^\times$ and an element $x \in k^\times$ such that $\omega(x) = 1$.
\end{Exs}

In general, maximal compact subgroups of a reductive group are parametrised by its enlarged Bruhat-Tits building \cite[3.2]{TitsCorvallis} (the building in \cite{TitsCorvallis} corresponds to the enlarged building \cite[4.2.6]{BruhatTits2}; see \cite[II.2]{RousseauHdR} for more details with bounded subgroups).

In fact, the additive group is the prototype of an algebraic group which does not have a maximal compact subgroup in its rational points. More precisely:

\begin{Thm}\label{thm:first:compact:condition}
Let $k$ be a non-Archimedean local field and $G$ a connected algebraic $k$-group.
The topological group $G(k)$ admits a maximal compact subgroup if, and only if, $G$ does not contain a non-trivial connected unipotent $k$-split normal $k$-subgroup.

Under these conditions, $G(k)$ is, moreover, compactly generated.
\end{Thm}

We will go back to the notion of splitness for unipotent groups; it corresponds to the existence of a filtration with subgroups isomorphic to $\mathbb{G}_a$.
In characteristic zero, all unipotent groups are split and, in fact, the above algebraic condition amounts to requiring that $G$ be reductive.
In this case, the theorem appears in \cite[§3.3]{PlatonovRapinchuk}.
Here, our theorem covers all cases and the proof, using Bruhat-Tits theory and pseudo-reductive groups, is uniform whatever the characteristic of the local field.

\subsection{Conjugacy and description of maximal pro-\texorpdfstring{$p$}{p} subgroups}

Once we know that an algebraic group $G$ admits maximal profinite subgroups (which are exactly maximal compact subgroups), we would like to describe them more precisely.
In the case of a semisimple $k$-group $G$, we can use the integral models of $G$ and the action of $G(k)$ on its Bruhat-Tits building $X(G,k)$.
There are, in general, several conjugacy classes of maximal profinite subgroups (in the simply connected case, they correspond to the different types of vertices).
However, the maximal pro-$p$ subgroups appear, in turn, to take the role of $p$-Sylow subgroups, as the following states:

\begin{Thm}\label{thm:conjugaison:maximal:pro-p}
Let $k$ be a non-Archimedean local field of residue characteristic $p$.
Let $G$ be a semisimple $k$-group.
Then, $G(k)$ admits maximal pro-$p$ subgroups and they are pairwise conjugate.

\end{Thm}

These maximal pro-$p$ subgroups will be described, in Theorem \ref{thm:description:models:pro-p}, thanks to some well-chosen integral models.

\subsection{Algebraic groups over imperfect fields}

As already mentioned, we have to use the notion of a pseudo-reductive group.
This notion was first introduced by Borel and Tits in \cite{BorelTitsNoteauCRAS} but was deeply studied only recently, by Conrad, Gabber and Prasad in \cite{CGP}.

If $k$ is any field, the unipotent radical of a smooth affine algebraic $k$-group $G$, denoted by $\mathcal{R}_{u,\overline{k}}(G_{\overline{k}})$, can fail to descend to a $k$-subgroup of $G$ when $k$ is imperfect.
It has a minimal field of definition which is a finite purely inseparable finite extension of the base field $k$ \cite[1.1.9]{CGP}.
Hence, we have to replace the unipotent radical ${\mathcal{R}_{u,\overline{k}}(G_{\overline{k}})}$ by the unipotent $k$-radical, denoted by $\mathcal{R}_{u,k}(G)$ and defined as the maximal smooth connected unipotent normal $k$-subgroup of $G$.
However, thanks to the following short exact sequence of algebraic $k$-groups:
\begin{equation*}
1 \rightarrow \mathcal{R}_{u,k}(G) \rightarrow G \rightarrow G / \mathcal{R}_{u,k}(G) \rightarrow 1
\end{equation*}
we can understand better the algebraic $k$-group $G$.
Of course, when $k$ is perfect, this is exactly the maximal reductive quotient of $G$.

Let $G$ be a smooth connected algebraic $k$-group.
One says that $G$ is \textbf{pseudo-reductive} if $\mathcal{R}_{u,k}(G)$ is trivial.
Over perfect fields, it corresponds to reductivity, but it is far from true in general. We have to face this difficulty because for a local field $k$ of characteristic $p$, we have $[k:k^p] = p$.

\begin{Ex}
If $k'/k$ is a purely inseparable finite extension, the commutative group $R_{k'/k}(\mathbb{G}_{m,k})$ is pseudo-reductive. Its unipotent radical $\mathcal{R}_{u,\overline{k}}(G_{\overline{k}})$ descent to a $k'$-group and is non-trivial since the quotient $R_{k'/k}(\mathbb{G}_{m,k}) / \mathbb{G}_{m,k}$ is unipotent \cite[1.1.3]{CGP}.
Therefore, it is not reductive.
\end{Ex}

Thanks to the main structural theorem of Conrad, Gabber and Prasad \cite[5.1.1]{CGP}, we have a deeper understanding of pseudo-reductive groups.
Hence, there is some hope to generalise results on reductive groups to pseudo-reductive groups and, by considering successive quotients, to obtain general results on arbitrary connected algebraic groups. Typically, this notion enabled B.~Conrad to obtain a Godement compactness criterion in terms of anisotropy for general groups over any local field (note that, until recently, standard references \cite{Margulis} quote this criterion for general groups in characteristic $0$ but only for reductive groups in positive characteristic, while it is now known to be true even for non-reductive groups in positive characteristic).

Thanks to the structure theory of unipotent groups provided by Tits \cite[B.2]{CGP}, we have notions of \gts{splitness}, \gts{isotropy} and \gts{anisotropy} for unipotent groups.
The most intriguing one is anisotropy, defined as follows.

Let $U$ be a smooth affine unipotent $k$-group.
One says that $U$ is $k$-\textbf{wound} if there are no non constant $k$-morphisms to $U$ from the affine $k$-line (where $U$ and $\mathbb{A}^1$ are seen as $k$-schemes), or equivalently if there is no nontrivial action of $\mathbb{G}_m$ on $U$.
Over a perfect base field, such a group has to be trivial; hence, this definition makes sense only for imperfect fields.

We recall the following definition of Bruhat and Tits \cite[1.1.12]{BruhatTits2}, initially introduced in a note of Borel and Tits \cite{BorelTitsNoteauCRAS}.

\begin{Def}\label{def:quasi:reductive}
Let $G$ be a smooth \textit{connected} algebraic $k$-group.
One says that $G$ is \textbf{quasi-reductive} if $\mathcal{R}_{u,k}(G)$ is $k$-wound.
\end{Def}

\begin{Rq}

Unless stated otherwise, we assume that a semisimple, reductive, pseudo-reductive or quasi-reductive $k$-group is connected by definition. Note that we have the containments:
$$\{\text{semisimple}\} \subset \{\text{reductive}\} \subset \{\text{pseudo-reductive}\} \subset \{\text{quasi-reductive}\}$$

In Theorem \ref{thm:first:compact:condition}, the algebraic $k$-group verifying the equivalent conditions are exactly the quasi-reductive ones.
\end{Rq}

In the same way as in the reductive case \cite[BTR theorem]{Prasad-TitsTheorem}, there is a correspondence between compactness and anisotropy for unipotent groups, given by Oesterlé \cite[VI.1]{Oesterle}: assume that $k$ is an imperfect local field, then $U$ is $k$-wound if, and only if, $U(k)$ is compact.

\subsection{The case of a topological base field}

From now on, $k$ is a local field of residual characteristic $p$.

If a totally disconnected locally compact group has a maximal compact subgroup, then its locally elliptic radical is compact (see Lemma \ref{lem:locally:elliptic:maximal}).
If $U$ is a non-trivial connected $k$-split unipotent $k$-group, we will build, in Lemma \ref{lem:unipotent:open:pro-p:covering} by analogy with the case of $\mathbb{G}_a$ seen in Example \ref{ex:compact:open:subgroups}, an exhaustion of the non-compact group $U(k)$ by (increasing) compact open subgroups.
Using Proposition \ref{prop:maximal:compact:existence:implies:quasireductive}, we deduce that if an algebraic $k$-group $G$ contains such a $U$ as a normal $k$-subgroup, then the elliptic radical of $G(k)$ is not compact since it contains $U(k)$ as a normal subgroup.
Thus, $G(k)$ cannot have a maximal compact subgroup.


Conversely, it is well-known that if $G$ is a semisimple $k$-group, then $G(k)$ has a maximal compact subgroup.
Hence, we would like to prove the same fact for any quasi-reductive $k$-group.
It is natural to exploit properness and finiteness properties of long exact sequences in Galois cohomology attached to some group extensions, but these properties are not satisfied in general.
In fact, the first Galois cohomology pointed sets of relevant normal subgroups of $G$ often fail to be finite in positive characteristic (e.g. $\# H^1(k,Z_G) = \infty$ when $\mathrm{char}(k) = p > 0$ and $G = \mathrm{SL}_p$; see also \cite[11.3.3]{CGP} for an example of a unipotent group).

Therefore cohomological methods are not sufficient to conclude.
We are using topological properties of rational points.
One of them is the following:

\begin{Def}\label{def:noetherian:group}
A topological group $G$ is called \textbf{Noetherian} if it satisfies the
ascending chain condition on open subgroups; this means that any increasing sequence of open subgroups of $G$ is eventually constant.
\end{Def}

\begin{Ex}\label{ex:noetherian:group}
(1) The discrete abelian group $(\ZZ,+)$ is Noetherian since any subgroup of $\ZZ$ is an ideal of the Noetherian ring $\ZZ$.

(2) By Example \ref{ex:compact:open:subgroups},
the additive group of a non-Archimedean local field is not a Noetherian group since it has an infinite strictly increasing sequence of open subgroups, namely $(\varpi^{-n} \mathcal{O}_k)_{n\in \NN}$.
\end{Ex}

Because the additive topological group $(k,+)$ (seen as the group of rational points of the additive group $\mathbb{G}_a$) admits no maximal compact subgroup, there is no hope for a non-$k$-wound unipotent group $U$ to have a maximal compact subgroup inside its rational points.
Together with Oesterl\'e's previously mentioned result, this is the heuristics leading to:

\begin{Thm}\label{thm:equivalence:quasi:reductive}
Let $k$ be a non-Archimedean local field with residue characteristic $p$ and $G$ be a smooth affine algebraic $k$-group. The following are equivalent:
\begin{enumerate}
\item[(i)] The identity component $G^0$ of $G$ is a quasi-reductive $k$-group,
\item[(ii)] $G(k)$ is Noetherian,
\item[(iii)] $G(k)$ admits a maximal compact subgroup,
\item[(iv)] $G(k)$ admits a maximal pro-$p$ subgroup.
\end{enumerate}

Moreover, under the above equivalent conditions:

(1) Every pro-$p$ (resp. compact) subgroup of $G(k)$ is contained in a maximal  pro-$p$ (resp. compact) subgroup of $G(k)$.

(2) Every maximal pro-$p$ (resp. compact) subgroup of $G(k)$ is open.
\end{Thm}

\begin{Cor}
If $G$ is a quasi-reductive $k$-group, then $G(k)$ is compactly generated.
\end{Cor}

\begin{proof}[Proof of corollary]
By \cite[Lemma 3.22]{CapraceMarquis} a locally compact group $G$ is Noetherian if, and only if, any open subgroup of $G$ is compactly generated.\qed
\end{proof}

This theorem and its corollary are well-known in the case of a $p$-adic field $k$ (in that case of $\mathrm{char}(k) = 0$, quasi-reductivity implies reductivity because all unipotent groups are split) as a proposition of Platonov and Rapinchuk \cite[3.3 Proposition 3.15]{PlatonovRapinchuk} and a theorem of Borel and Tits \cite[13.4]{BorelTits}.
In nonzero characteristic it is necessary to consider the notion of quasi-reductivity in the statement of the result.

For a reductive group $G$ defined over a $p$-adic field, we know moreover that a compact open subgroup is contained in finitely many compact subgroups \cite[Proposition 3.16 (1)]{PlatonovRapinchuk}.
We don't know if this statement is still true for a quasi-reductive group over a local field of positive characteristic.
In fact, when $G(k)$ acts properly on a locally finite affine building, there is a correspondence between its compact open subgroups and the non-empty bounded subsets of the Bruhat-Tits building.
In the quasi-reductive case, we have a spherical Tits system by \cite[C.2.20]{CGP} but the existence of an affine Tits system is not yet proven.

\begin{Ex}
Let $k'/k$ be a purely inseparable finite extension of local fields and $G'=\mathrm{GL}_{n,k'}$. Let $Z \simeq \mathbb{G}_{m,k}$ be the $k$-split torus canonically contained in the center $\mathcal{Z}_{G'}$ of $G'$.
Consider the $k$-group $G = R_{k'/k}(G') / Z$.
It is not pseudo-reductive since it contains the central unipotent $k$-subgroup $R_{k'/k}(\mathcal{Z}_{G'}) / Z$.
In particular, it is not the group of rational points of any reductive group.
By Hilbert 90 theorem and Lemma \ref{lem:exact:sequence:cohomology}, the quotient morphism provides an isomorphism of topological groups $G(k) \simeq G'(k') / Z(k)$.
Since $G'$ and $Z$ are reductive groups, there rational points are Noetherian groups by Theorem \ref{thm:equivalence:quasi:reductive}. By Proposition \ref{prop:noetherian:groups} (6), the group $G(k)$ is Noetherian as the image of a Noetherian group by a continuous map. By Theorem \ref{thm:equivalence:quasi:reductive}, we deduce that $G$ is quasi-reductive.
\end{Ex}

Standard presentations of pseudo-reductive group provides much more information on the structure of these groups and the proof of Theorem \ref{thm:equivalence:quasi:reductive} does not use all the aspects of that tool.
Once we know that the rational points of a pseudo-reductive group admits a maximal compact subgroup, locally elliptic subgroups (see Definition \ref{def:locally:elliptic}) appear in turn to be relevant.
We obtain the following result on the topological structure of rational points of a quasi-reductive group:

\begin{Thm}\label{thm:structure:quasi:reductive}
Let $G$ be a connected quasi-reductive group over a non-Archimedean local field $k$.
Any open subgroup $V$ of the group of rational points $G(k)$ admits a chain of closed normal subgroups $1 \leqslant Q_{V} \leqslant S_{V} \leqslant V$ such that
$Q_{V}$ is compact,
the quotient $S_{V}/Q_{V}$ is the internal direct product of finitely many non-compact, topologically simple, compactly generated locally compact groups
that are each isomorphic to the quotient of rational points of a simply connected isotropic simple algebraic group over a local field (of the same characteristic and residue characteristic as $k$) by its center,
and the quotient $V/S_{V}$ is compactly generated and virtually abelian.

Moreover, if $G$ is pseudo-reductive and $V=G(k)$, the compact group $Q_V$ is the locally elliptic radical of $G(k)$.
\end{Thm}

\subsection{Use of buildings and integral models}

Though Theorem \ref{thm:equivalence:quasi:reductive} gives a good criterion for the existence of maximal compact subgroups, the proof is not constructive in the sense that we do not have any detail about these subgroups.
Nevertheless, in the case of a semisimple $k$-group $G$, denote by $X(G,k)$ its Bruhat-Tits building.
In Proposition \ref{prop:description:compact:maximal}, we get a good description of maximal compact subgroups as stabilizers of some points for the continuous action of $G(k)$ on its Bruhat-Tits building.

As stated in Theorem \ref{thm:equivalence:quasi:reductive}, for a semisimple $k$-group $G$, the topological group $G(k)$ has maximal pro-$p$ subgroups.
These groups are a kind of generalisation of Sylow subgroups for a finite group: in the profinite situation, a profinite group has maximal pro-$p$ subgroups and they are pairwise conjugate \cite[1.4 Prop. 3]{SerreCohomologieGaloisienne}.
By our second main theorem \ref{thm:conjugaison:maximal:pro-p}, we know that the (usually non-compact) group $G(k)$ has maximal pro-$p$ subgroups and that they are pairwise conjugate.
The use of Bruhat-Tits buildings and, in particular, of Euclidean buildings associated to pairs $(G,k)$ allows us to be more precise:
we give a useful description of maximal pro-$p$ subgroups by use of a valued root groups datum in the simply-connected case.
Thanks to this, in a further work \cite{Loisel-GenerationFrattini}, we compute the Frattini subgroup of a maximal pro-$p$ subgroup.
There will be a somewhat analogous computation as in \cite{PrasadRaghunathan1} where Prasad and Raghunathan compute the commutator subgroup of a parahoric subgroup.

\begin{Thm}\label{thm:description:maximal:pro-p}
Let $k$ be a non-Archimedean local field and $G$ a connected semisimple $k$-group.
If $P$ is a subgroup of $G(k)$,
then $P$ is a maximal pro-$p$ subgroup of $G(k)$ if, and only if, there exists an alcove $\mathbf{c} \subset X(G,k)$ such that $P$ is a maximal pro-$p$ subgroup of the stabilizer of $\mathbf{c}$.

Moreover, such an alcove $\mathbf{c}$ is uniquely determined by $P$ and the set of fixed points by $P$ in $X(G,k)$ is contained in the simplicial closure $\mathop{cl}(\mathbf{c})$ of $\mathbf{c}$.

In particular, there is a natural surjective map from the maximal pro-$p$ subgroups of $G(k)$ to the alcoves of $X(G,k)$. When $G$ is simply connected, this map is a bijection.
\end{Thm}

The first part of this theorem is a direct consequence of Proposition \ref{prop:pro-p:stabilises:alcove} and conjugation of $p$-Sylow subgroups in profinite groups since the stabilizer of an alcove is a profinite group by Lemma \ref{lem:parahoric:compact}(2).
To get a deeper description of maximal pro-$p$ subgroups, integral models and their reductions are useful.

\begin{Not}
Let $\Omega \subset \mathbb{A}$ be a non-empty bounded subset where $\mathbb{A}$ denotes the standard apartment of the Bruhat-Tits building $X(G,k)$.
Denote by $\mathfrak{G}_\Omega$ the corresponding smooth connected affine $\mathcal{O}_k$-model of $G$ (denoted by $\mathfrak{G}_\Omega^\circ$ in \cite{BruhatTits2} and by $\mathfrak{G}_\Omega$ in \cite{Landvogt}: they are the same $\mathcal{O}_k$-model of $G$, up to isomorphism, because they satisfy the same universal property).
Denote by $\mathfrak{G}_\Omega^\dagger$ the (possibly non-connected) smooth affine $\mathcal{O}_k$-model defined in \cite[4.6.18]{BruhatTits2} for the quasi-split case and, by descent, in \cite[5.1.8]{BruhatTits2} for the general case.
\end{Not}

Recall that if $\Omega$ satisfies a suitable notion of convexity as a subset of a polysimplicial structure (denote by $\mathop{cl} (\Omega)$ the simplicial closure defined in \cite[7.1.2]{BruhatTits1}, we assume here that $\Omega = \mathop{cl} (\Omega)$) and $G$ is semisimple, then $\mathfrak{G}_\Omega^\dagger(\mathcal{O}_k)$ is the stabilizer of $\Omega$ in $G(k)$ \cite[4.6.29, 5.1.31]{BruhatTits2}.
The group $\mathfrak{G}_\Omega(\mathcal{O}_k)$ fixes $\Omega$ pointwise and, when $G$ is simply-connected we have $\mathfrak{G}_\Omega = \mathfrak{G}_\Omega^\dagger$ \cite[5.2.9]{BruhatTits2}.
In particular, a simply-connected semisimple $k$-group acts on its Bruhat-Tits building by type-preserving isometries.

In part \ref{subsection:description:integral:models}, we will use $\mathcal{O}_k$-models (where $\mathcal{O}_k$ denotes the ring of integers of $k$) to get the following description:

\begin{Thm} \label{thm:description:models:pro-p}
Let $k$ be a non-Archimedean local field and $G$ a connected simply connected semisimple $k$-group.

A maximal pro-$p$ subgroup of $G(k)$ is conjugate to
\begin{equation*}
P_{\mathbf{c}}^+ = \mathrm{ker} \Big( \mathfrak{G}_{\mathbf{c}}(\mathcal{O}_k ) \twoheadrightarrow \overline{\mathfrak{G}}_{\mathbf{c}}^{\mathrm{red}}( \kappa ) \Big)
\end{equation*}
where $\mathbf{c} \subset \mathbb{A}$ denotes an alcove of the standard apartment,
$\kappa$ denotes the residue field of $k$ and $\overline{\mathfrak{G}}_{\mathbf{c}}^{\mathrm{red}}$ denotes the reductive quotient of the special fiber of the integral model associated to $\mathbf{c}$.
\end{Thm}

This morphism $\mathfrak{G}_{\mathbf{c}}(\mathcal{O}_k ) \twoheadrightarrow \overline{\mathfrak{G}}_{\mathbf{c}}^{\mathrm{red}}( \kappa )$ and its kernel appear in several references like \cite{TitsCorvallis}.

\subsection{Acknowledgements}

The author is grateful to Bertrand R\'emy and the anonymous reviewers for the careful reading they have done of this text; and Philippe Gille, Brian Conrad and Giancarlo Lucchini-Arteche for their remarks.

The author also would like to warmly thank an anonymous referee for their comments and suggestions that led to Proposition \ref{prop:structure:pseudo:reductive} and Theorem \ref{thm:structure:quasi:reductive}.

The author is supported by the project ANR Geolie, ANR-15-CE40-0012, (The French National Research Agency).

\section{Maximal compact subgroups}

\subsection{Extensions of topological groups}

As we consider topological groups, we require that any morphism between such groups be continuous. Recall that the morphism deduced from an algebraic morphism is always continuous for the $k$-topology.

\subsubsection*{Noetherian groups}

Firstly, let us recall some properties of Noetherian groups (Definition \ref{def:noetherian:group}).

\begin{Prop}\label{prop:noetherian:groups}

\begin{enumerate}
\item[]

\item[(1)] Any open subgroup of a Noetherian group is Noetherian.

\item[(2)] A compact group is Noetherian.

\item[(3)] Let $\varphi : G \rightarrow Q$ a strict (continuous) morphism between topological groups with open image (e.g. $\varphi$ is an open morphism). If $Q$ and $\ker \varphi$ are Noetherian, then so is $G$.

\item[(4)] Any extension of Noetherian groups is a Noetherian group.

\item[(5)] The multiplicative group $k^\times$ of a non-Archimedean local field $k$ is Noetherian.

\item[(6)] Let $\psi : H \rightarrow G$ a (continuous) morphism between topological groups. If $H$ is Noetherian and $\psi(H)$ is a finite-index normal subgroup of $G$, then $G$ is Noetherian.
\end{enumerate}
\end{Prop}

\begin{proof}
(1) is obvious.

(2) is clear since an open subgroup of a compact group has finite index.

(3) Since $\mathrm{Im}(\varphi)$ is open in $Q$, the subgroup $\varphi(G)$ is Noetherian by (1).
Since $\varphi$ is a strict morphism, we may and do assume that $\varphi$ is the quotient map ${G \rightarrow G / H \simeq \varphi(G)}$ where $H = \ker \varphi$.
Let $(U_n)_n$ an increasing sequence of open subgroups of $G$.
Since $H$ is Noetherian, the sequence $(U_n \cap H)_n$ is eventually constant, say from $N_1 \in \NN$.
Moreover, the sequence $\varphi(U_n) \simeq U_n H / H$ is eventually constant, say from $N_2 \geq N_1$, since $\varphi(U_n)$ is open in the Noetherian group ${\varphi(G) \simeq G / H}$.
We compute $\varphi(U_n) \simeq U_n / (U_n \cap H) \simeq U_n / (U_{N_1} \cap H) \simeq U_{N_2} / (U_{N_1} \cap H)$ for all $n \geq N_2$.
Hence $U_n = U_{N_2}$ for all $n \geq N_2$.

(4) By definition, an extension of topological groups is an exact sequence
\begin{equation*}
1 \rightarrow H \stackrel{j}{\rightarrow} G \stackrel{\pi}{\rightarrow} Q \rightarrow 1
\end{equation*}
of continuous morphisms which are open on their image.
Applying (3) to $\pi$, if $H$ and $Q$ are Noetherian, then so is $G$.

(5) is a consequence of (2) and (4) since $k^\times$ is an extension of the compact subgroup $\mathcal{O}_k^\times$ by the Noetherian discrete group $\omega(k^\times) \simeq \ZZ$.

(6) Let $(U_n)_n$ an increasing sequence of open subgroups of $G$.
Since $H$ is Noetherian, the sequence of open subgroups $\psi^{-1}(U_n)$ is eventually constant, and so is the sequence $V_n = \psi(\psi^{-1}(U_n)) = U_n \cap \psi(H)$.
The sequence of indices $[U_n:V_n] = [U_n:U_n \cap \psi(H)]$ is a sequence of integers bounded by the finite index $[G:\psi(H)]$.
Moreover, since $U_n$ is an increasing sequence and $V_n$ is eventually constant, the sequence $[U_n:V_n]$ is eventually increasing, hence eventually constant.
As a consequence, the increasing sequence $\left( U_n \right)_n$ is eventually constant.\qed
\end{proof}

\begin{Rq}\label{rq:existence:required}
A motivation to consider the Noetherian property on topological groups is that one can easily prove the existence of maximal subgroups with a given property $(P)$, as soon as we know the existence of some open subgroup satisfying the desired property $(P)$ (like in proof of \ref{prop:noetherian:implies:maximal}).

As an example, a Noetherian group with a proper open subgroup has maximal proper open subgroups, and any proper open subgroup is contained in, at least, one of them.
\end{Rq}

\subsubsection*{Morphisms of $k$-scheme and an exact sequence}

Secondly, let us recall some properties of algebraic morphisms between topological groups of rational points.

\begin{Lem}\label{lem:exact:sequence:cohomology}
Let $k$ be a non-Archimedean local field.
Let $G$ be a smooth affine algebraic $k$-group and $H$ a normal closed $k$-subgroup of $G$.

(a) There exists a faithfully flat quotient homomorphism $\pi : G \rightarrow G/H$  where $G/H$ is a smooth $k$-group.
Moreover, when $H$ is smooth, $\pi$ is smooth.

(b) The following exact sequence:
\begin{equation*}
1 \rightarrow H \overset{j}{\rightarrow} G \overset{\pi}{\rightarrow} G/H \rightarrow 1
\end{equation*}
induces an exact sequence of topological groups:
\begin{equation*}
1 \rightarrow H(k) \stackrel{j_k}{\rightarrow} G(k) \stackrel{\pi_k}{\rightarrow} (G/H)(k)
\end{equation*}
and $j_k$ is a homeomorphism onto its image.
Moreover, if $H$ is smooth, then the continuous morphism $\pi_k$ is open.

\end{Lem}

\begin{proof}
(a) The quotient morphism exists and is faithfully flat by \cite[Exp. VI A Thm 3.2 (iv)]{SGA3}.
Hence, the $k$-group $G/H$ is smooth \cite[II.§5 2.2]{DemazureGabriel}.
If, moreover $H$ is smooth, by \cite[II.§5 5.3 and II.§5 2.2]{DemazureGabriel}, the morphism $\pi$ is smooth.

(b) Morphism between $k$-schemes of finite type are continuous for the $k$-topology, and $j_k$ is a homeomorphism onto its image by definition of the $k$-topology.
Since $\pi$ is smooth, the continuous morphism $\pi_k$ is open by \cite[lemma 3.1.2 and proposition 3.1.4]{GabberGilleMoretBailly}.\qed

\end{proof}

\subsubsection*{Existence of a pro-$p$ open subgroup}

By the Remark \ref{rq:existence:required}, we need and recall the following lemma:

\begin{Lem}\label{lem:existence:open:pro-p}
Let $k$ be a non-Archimedean local field of residual characteristic $p$.
Let $G$ be a smooth affine algebraic $k$-group.
Then $G(k)$ contains a pro-$p$ open subgroup.
\end{Lem}

\begin{proof}
Given a closed immersion $G \rightarrow \mathrm{GL}_{n,k}$ (such an immersion exists \cite[II.5.5.2]{DemazureGabriel}), the topological group $G(k)$ can be seen as a closed subgroup ${G(k) \subset \mathrm{GL}_n(k)}$ endowed with the usual topology.
Hence, it is sufficient to prove that $\mathrm{GL}_n(k)$ contains a pro-$p$ open subgroup $U$, since $U \cap G(k)$ will be a pro-$p$ open subgroup of $G(k)$.

The group $H = \mathrm{GL}_n(\mathcal{O}_k)$ is profinite since it is a totally disconnected compact group.
For $d \in \NN^*$, define
$$H_d = \mathrm{GL}_n(\mathfrak{m}^d) = \left\{g \in \mathrm{GL}_n(\mathcal{O}_k)\,,\,g - \mathrm{id} \in \mathfrak{m}_K^d \mathcal{M}_n(\mathcal{O}_k) \right\}$$

The $H_d$ are normal compact open subgroups of $\mathrm{GL}_n(\mathcal{O}_k)$, and form a basis of open neighbourhoods of $\mathrm{id} \in H$.
Moreover, they are pro-$p$ groups in the same way as \cite[5.1]{DixonDuSautoyMannSegal} for $\mathrm{GL}_n(\mathbb{Z}_p)$.

\textbf{Claim:} $H_1 = \varprojlim_{d} H_1 / H_d$ is a pro-$p$-group.

For any $x \in H_d$, write $x = \mathrm{id} + y$ where $y \in \mathfrak{m}^d \mathcal{M}_n(\mathcal{O}_k)$.
Hence $x^p = \mathrm{id} + p y + \sum_{k=2}^p \binom{p}{k} y^k$.
If $k \geq 2$, then $y^k \in \mathfrak{m}_K^{d+1} \mathcal{M}_n(\mathcal{O}_k)$ because $d \geq 1$.
If $\mathrm{char}(k) = p$, then $py = 0$.
Else, $\mathrm{char}(k) = 0$ and $p \in \mathfrak{m}$.
Hence $p y \in \mathfrak{m}^{d+1} \mathcal{M}_n(\mathcal{O}_k)$, so $H_d / H_{d+1}$ is a $p$-group.\qed
\end{proof}

\subsection{Compact and open subgroups of a semisimple group}

In this section, we assume that $G$ is an affine smooth connected semisimple $k$-group where $k$ is a non-Archimedean local field.
In Proposition \ref{prop:description:compact:maximal}, we describe maximal compact subgroups as stabilizers of, uniquely defined, points of the building. This is still true if we only assume that $G$ is reductive.
We do not assume, in general, that $G$ is simply connected and some consequences of this additional assumption will be given.
Such a group $G(k)$ acts continuously and strongly transitively on its affine Bruhat-Tits building (with a type-preserving action when $G$ is, moreover, simply connected).
We denote by $\mathbb{A}$ the standard apartment, by $\mathbf{c}$ a chosen alcove in $\mathbb{A}$ and by $G^+$ the subgroup of $G(k)$ consisting of the type-preserving elements.

Define $B = \mathrm{Stab}_{G(k)}(\mathbf{c})$ the setwise stabilizer of $\mathbf{c}$ and $B^+ = \mathrm{Stab}_{G^+}(\mathbf{c})$ the pointwise stabilizer of $\mathbf{c}$.
Define $N = \mathrm{Stab}_{G(k)}(\mathbb{A})$ the setwise stabilizer of $\mathbb{A}$ in $G(k)$ and $N^+ = \mathrm{Stab}_{G^+}(\mathbb{A})$ the setwise stabilizer of $\mathbb{A}$ in ${G}^+$.
Thus, $(B,N)$ is a generalised BN-pair of $G(k)$ (see \cite[5.5 and 14.7]{Garrett} for details).
Define ${T = B \cap N}$ and ${T^+ = B^+ \cap N^+}$,
and put ${W}^+ = {N}^+ / {T}^+$.
The set $\Theta = T / {T}^+$ is finite \cite[5.5]{Garrett} and we have a Bruhat decomposition $\displaystyle G(k) = \bigsqcup_{t \in \Theta\,,\,w \in {W}^+} {B}^+ t w {B}^+$.
Define the following bornology on $G(k)$ by:

\begin{Def,Prop}[from {\cite[14.7]{Garrett}}]\label{def:bounded:subset}
A subset $H \subset G(k)$ is called \textbf{bounded} if $H$ satisfies the following equivalent properties:
\begin{enumerate}
\item[(i)] $H$ is contained in a finite union of double cosets $B^+ t w B^+$,
where $t \in \Theta$ and $w \in {W}^+$,
\item[(ii)] there exists a point $x \in X(G,k)$ such that $H \cdot x \subset X(G,k)$ is bounded,
\item[(iii)] for any bounded subset $Y \subset X(G)$, the subset $H \cdot Y = \{ h \cdot y\,,\,h\in H \text{ and }y \in Y \} \subset X(G,k)$ is bounded.
\end{enumerate}
\end{Def,Prop}

Given an embedding $G(k) \rightarrow \mathrm{GL}_n(k)$, there is a natural definition of bounded subsets, provided by the canonical metric on $\mathrm{GL}_n(k)$.
One can note that both definitions coincide. 

\begin{Lem}\label{lem:parahoric:compact}
Under the above assumptions and notations:
\begin{enumerate}

\item[(1)] For any non-empty subset $\Omega \subset \mathbb{A}$, the pointwise stabilizer of $\Omega$ in $G(k)$ is compact.

\item[(2)] For any non-empty bounded subset $\Omega \subset X(G,k)$, the setwise stabilizer of $\Omega$ in $G(k)$ is compact.

\end{enumerate}
\end{Lem}

\begin{proof}
Since $G$ is semisimple, we know by \cite[2.2]{TitsCorvallis} that $G(k)$ acts properly on $X(G,k)$ that is locally finite.

(1) In particular, the stabilizer $P_x$ of a point $x \in \mathbb{A}$ is compact (by construction).
Hence, the pointwise stabilizer of $\Omega$ written $P_\Omega = \bigcap_{x \in \Omega} P_x$ \cite[13.3(i) and 12.8]{Landvogt} is compact.

(2) If $x \in X(G,k)$, then there exists $g \in G(k)$ such that $g \cdot x \in \mathbb{A}$ and it gives $\mathrm{Stab}_{G(k)}(x) = g^{-1} P_{g \cdot x} g$.
This does not depend on the choice of such a $g \in G(k)$.
Consider $\Omega \subset X(G,k)$ a non-empty subset bounded subset.
It is finite as cellular complex since $X(G,k)$ is locally finite).
By (2), the pointwise stabilizer of $\Omega$ is compact.
Since $\Omega$ is finite, its setwise stabilizer is compact.\qed
\end{proof}

As a consequence of this lemma, bounded subsets are closely linked to compact subsets.

\begin{Lem}\label{lem:bounded:versus:compact}
Under the above assumptions and notations:

\begin{enumerate}
\item[(1)] Every bounded subset of $G(k)$ is relatively compact.
\item[(2)] A subset of $G(k)$ is compact if, and only if, it is closed and bounded.
\item[(3)] Every maximal bounded subgroup of $G(k)$ is a maximal compact subgroup. 
\end{enumerate}
\end{Lem}

\begin{proof}
Recall that $B^+ = P_{\mathbf{c}}$ is compact by Lemma \ref{lem:parahoric:compact} and open in $G(k)$ by \cite[12.12 (ii)]{Landvogt}.
Hence, every double coset $B^+ t w B^+$ is a compact open subset of $G(k)$.

(1) If $H \subset G(k)$ is bounded, then by Definition \ref{def:bounded:subset}(i) $H$ is contained in a finite union of double cosets, and this union is a compact subset.

(2) If $H$ is a compact subset of $G(k)$, then $H$ is closed in $G(k)$.
The open cover of $H$ by double cosets has a finite subcovering.
By Definition \ref{def:bounded:subset}(i), $H$ is bounded.
Conversely, a bounded subset is compact when it is closed, by (1).

(3) If $H$ is a maximal bounded subgroup, then $\overline{H}$ is a closed subgroup.
It is bounded by Definition \ref{def:bounded:subset}(ii) and contains $H$.
Hence, maximality of $H$ implies $H = \overline{H}$ is a maximal compact subgroup, because every compact subgroup is bounded according to (2).\qed
\end{proof}

Recall that a metric space is said to be $\mathrm{CAT}(0)$ if it is geodesic (any two points are connected by a continuous path parametrized by distance) and if any geodesic triangle is at least as thin as in the Euclidean plane (for the same edge lengths).
This notion is developed in the book of Bridson and Haefliger \cite{BridsonHaefliger}.
The latter condition is a non-positive curvature one (called also $\mathrm{(NC)}$ in \cite[§VI.3B]{Brown}), which can also be formulated by requiring the parallelogram inequality \cite[Prop. 11.4]{AbramenkoBrown}.
We use the following fixed-point theorem to describe compact open subgroups thanks to the metric space $X(G,k)$.

\begin{Thm}[{Bruhat-Tits fixed point theorem \cite[VI.4]{Brown}}]\label{thm:fixed:point:theorem}
Let $H$ be a group acting isometrically on a complete CAT(0) metric space $(M,d)$.
If $M$ has a $H$-stable non-empty bounded subset, then $H$ fixes a point in $M$.
\end{Thm}

The following corollary is a immediate consequence of the fixed point theorem and the Definition \ref{def:bounded:subset}(iii).

\begin{Cor}\label{cor:bounded:fixed:point}
If $H$ is a bounded subgroup of $G(k)$, then $H$ fixes a point of $X(G,k)$.
\end{Cor}

Let us give a proof of the following proposition:

\begin{Prop}\label{prop:description:compact:maximal}
Let $k$ be a non-Archimedean local field and $G$ a semisimple $k$-group.
Let $P$ be a subgroup of $G(k)$.
The following are equivalent:
\begin{enumerate}
\item[(i)] the subgroup $P$ is a maximal compact subgroup of $G(k)$,
\item[(ii)] the subgroup $P$ fixes a unique point $x \in X(G,k)$ and $P = \mathrm{Stab}_{G(k)}(x)$.
\end{enumerate}

Moreover, if $G$ is simply connected, such an $x$ is a vertex in the simplicial complex $X(G,k)$.
\end{Prop}

\begin{proof}
$(i) \Rightarrow (ii)$
If $P$ is a maximal bounded subgroup, then, by Corollary \ref{cor:bounded:fixed:point}, $P$ fixes a point $x \in X(G,k)$.
Hence, $P$ is a subgroup of $\mathrm{Stab}_{G(k)}(x)$ which is a bounded subgroup by \ref{lem:parahoric:compact}(2).
We get $P = \mathrm{Stab}_{G(k)}(x)$ because of the maximality assumption on $P$.

We have to show that the maximality of $P$ implies that $\mathrm{Stab}_{G(k)}(x)$ does not fix any other point of $X(G,k)$.
Let $A$ be an apartment containing $x$.
Denote by $\mathcal{H}$ the set of walls in $A$.
Let us prove that the maximal bounded subgroup $\mathrm{Stab}_{G(k)}(x)$ has a unique fixed point in $A$.

By transitivity, there exists an element $g \in G(k)$ such that $g \cdot A = \mathbb{A}$. Hence, we have to show that $g \mathrm{Stab}_{G(k)}(x) g^{-1} = P_{g\cdot x}$ fixes a unique point of $\mathbb{A}$.
We can and do assume that $x\in \mathbb{A}$ and $g=1$.

Denote by $F$ the facet containing $x$.
By contradiction, assume that $P_x$ fixes another point $y \in \mathbb{A}$.
Then we have $P_x \subset P_y$ and this should be an equality by maximality of $P_x$.
If $y \not\in F$, then there exists a wall $H \in \mathcal{H}$ such that $y \in H$ and $x \not\in H$. Consider the affine root $a+l$ such that the associated half-apartment $D(a,l)$, with boundary $H$, does not contain $x$.
The set of points of $\mathbb{A}$ fixed by the root group $U_{a,l}$ is exactly the half-apartment $D(a,l)$ \cite[13.3 (ii)]{Landvogt}.
The root group $U_{a,l}$ fixes $y$ but not $x$, and is contained in $P_y$.
Therefore, the inclusion $P_x \subset P_y$ is not an equality and we get a contradiction.
Hence, $y$ and $x$ have to be on the same facet $F$.
The action being isometric and $[x,y]$ being metrically characterized \cite[Prop. 11.5]{AbramenkoBrown}, $P_x$ fixes the line segment $[x,y]$.
If $x \neq y$, since the action is continuous and preserves the polysimplicial structure, the group $P_x$ fixes $\overline{F} \cap (x,y)$.
Hence $P$ fixes a point $z \in \overline{F} \setminus F$.
We obtain a contradiction by considering the fixed point $z \not\in F$.

$(ii) \Rightarrow (i)$
Conversely, let $x \in X(G,k)$ be such that the group $P = \mathrm{Stab}_{G(k)}(x)$ has a unique fixed point.
If $P'$ is a bounded subgroup containing $P$, and ${y \in X(G,k)}$ a point fixed by $P'$, then $P$ fixes $y$ and $y = x$ because of uniqueness.
Hence $P' \subset \mathrm{Stab}_{G(k)}(x) = P$.

Moreover, if $G$ is simply connected, the stabilizer of a facet fixes it pointwise \cite[5.2.9]{BruhatTits2}.
Because of the above equivalence, a maximal bounded subgroup is exactly the stabilizer of a vertex of $X(G,k)$.\qed
\end{proof}

\begin{Rq}\label{rq:isolated:maximal:bounded}
By uniqueness of the fixed point, we get an injective map from the set of maximal bounded subgroups of $G(k)$ to the set of points in $X(G,k)$.
Denote by $X(G)_{\max}$ the image of this map.
It is easy to remark that $X(G,k)_{\max}$ contains the vertices of the polysimplicial complex $X(G,k)$.

Moreover, it is easy to see that every $x \in X(G,k)_{\max}$ is the center of mass of its facet $F$, because the stabilizer in $G(k)$ of $x$ acts by isometries on $F$ and $x$ is the only fixed point.
We emphasize that the converse is not true: the stabilizer of the center of mass of a facet is not a maximal bounded subgroup in general.
Indeed, the case of simply connected groups provides a counter-example by considering the center of mass of a facet that is not a vertex by Proposition \ref{prop:description:compact:maximal}.
\end{Rq}

\begin{Rq}\label{rq:compact:contained:in:maximal}
Using the proof of Proposition \ref{prop:description:compact:maximal}, it is not hard to see that a compact subgroup $H \subset G(k)$ is always contained in a maximal one.
Consider a fixed point $x \in X(G,k)$ by $H$ of maximal degree $d(x)$ (this does not depends on the choice of an apartment).
Hence, $H$ is contained in $\mathrm{Stab}_{G(k)}(x)$.

Claim: $\mathrm{Stab}_{G(k)}(x)$ is a maximal compact subgroup.

By contradiction, if $\mathrm{Stab}_{G(k)}(x)$ is not, then it fixes a second point $y$, and then one can find a fixed point on the line $(x,y)$ of higher degree: this contradicts the maximality of $d(x)$.
\end{Rq}

Now, we need further investigation on compact open subgroups to prove Noetherianity for absolutely simple semisimple groups.

\begin{Lem}\label{lem:compact:open:fixed:points}
Let $U$ be a compact subgroup of $G(k)$ and denote by $\Omega = X(G,k)^U$ the non-empty subset of points fixed by $U$.
If $U$ is open, then $\Omega$ is a bounded (therefore compact) subset of $X(G,k)$.
\end{Lem}

\begin{proof}
By contradiction, assume that $\Omega$ is not bounded.
Let $x_0 \in \Omega$.
Since $\Omega$ is not bounded, one can choose a sequence $x_n \in \Omega$ such that $d(x_n,x_0) \geq n$.
Let $\overline{X(G,k)}$ be a compactification of $X(G,k)$, defined in \cite{RemyThuillierWerner}.
Let $x \in \overline{X(G,k)}$ be a limit point of $(x_n)_n$ (it exists because $\overline{X(G,k)}$ is a compact space by \cite[3.34]{RemyThuillierWerner}).
Because $\begin{array}{ccc}X(G,k)&\rightarrow&\RR\\y&\mapsto&d(x_0,y)\end{array}$ is continuous, $x \not\in X(G,k)$.
By \cite[4.20 (i)]{RemyThuillierWerner}, there exists a maximal $k$-split torus $S'$ such that $x_0,x \in \overline{\mathbb{A}(S',k)}$, and one can assume that $S' = S$.

The group $U$ is open in the subgroup $\mathrm{Stab}_{G(k)}(x_0)$, which is compact by \ref{lem:parahoric:compact}.
Hence, for every relative root $a \in \Phi$, the intersection $U_{a}(k) \cap U$ has finite index in the subgroup $U_{a,{x_0}}$.
Hence, $U$ contains $U_{a,l_a}$ for some $l_a \in [f_{x_0}(a),+ \infty[$.

Because $G(k)$ acts continuously on $\overline{X(G,k)}$, the point $x \in \overline{X(G,k)}$ is fixed by $U$.
With the notations of \cite[§4.1]{RemyThuillierWerner}, since $x \not\in \mathbb{A}(S,k)$, it belongs to a stratum $\mathcal{B}(Q_{\text{ss}},k)$ for some proper $k$-parabolic subgroup $Q$.
By \cite[4.12(ii)]{RemyThuillierWerner}, we know that $\operatorname{Stab} t_{G}(x)(k)$ is Zariski-dense in $Q$.
Since $Q$ is a proper $k$-parabolic subgroup, there exists a root $a \in \Phi(G,S)$ such that $U_a \cap Q = 1$.
Hence $U_{a,x} = {\mathrm{Stab}_{G(k)}(x) \cap U_{a}(k)} = \{1\}$.
But $U_{a,x} \supset {U \cap U_{a}(k)} \supset U_{a,l_a}$, and we get a contradiction.\qed
\end{proof}

Let us give another proof using the visual boundary of the building instead of a compactification.

\begin{proof}
If $X(G,k)$ is a single point, the lemma is obvious. Assume that $X(G,k)$ is not a single point.
For any minimal $k$-parabolic subgroup $P$ of $G$, we know that its unipotent radical $\mathcal{R}_u(P)$ is defined over $k$ \cite[20.5]{Borel} and directly spanned by some root groups \cite[21.11]{Borel}.
By \cite[13.3]{Landvogt}, we know that any element of a root group fixes an half apartment, therefore its action on the building preserves the type of facets.
Let $G(k)^+$ be the subgroup of $G(k)$ generated by the rational points subgroups $\mathcal{R}_u(P)(k)$ of the unipotent radical of minimal parabolic subgroups of $G$.
Then the action of $G(k)^+$ on the Bruhat-Tits building $X(G,k)$ preserves the types of facets.

Let $S$ be a maximal $k$-split torus and let $A_S = A(G,S,k)$ be the apartment defined from $S$.
Because $G(k)$ acts strongly transitively on $X(G,k)$ and $N = \mathcal{N}_G(S)(k)$ is the stabilizer of $A_S$ \cite[13.8]{Landvogt}, we know that $N$ acts transitively on the set of alcoves of $A_S$.
Because $N$ acts on $A_S$ through a group homomorphism $N \to V \rtimes W \subset \mathrm{Aff}(A_S)$ \cite[1.8]{Landvogt}, one can consider the subgroup $N'$ of $N$ consisting in elements that preserve the type of faces.
One can show that $N'$ acts transitively on the set of alcoves of $A_S$ because any element of the spherical Weyl group (corresponding to the Weyl group of a special vertex) can be lift in $N$ \cite[12.1(ii)]{Landvogt}.
Consider the subgroup $H$ of $G(k)$ generated by all those $N'$ and $G(k)^+$.
Then $H$ acts on $X(G,k)$ by type-preserving isometries.
We claim that $H$ acts strongly transitively on $X(G,k)$.
Indeed, on the one hand, $H$ acts transitively on the set of alcoves contained in a given apartment and, on the other hand, by \cite[5.1.31]{BruhatTits2}, we know that $H$ acts transitively on the set of apartments containing a given alcove.

Denote by $\partial X(G,k)$ the visual boundary of $X(G,k)$ (as defined in \cite[16.9]{Garrett} or \cite[11.9]{AbramenkoBrown}).
Since $H$ acts strongly transitively on $X(G,k)$ by type-preserving isometries, we know by \cite[17.1]{Garrett} that $H$ acts strongly transitively by type-preserving isomorphisms on $\partial X(G,k)$.
Moreover, by \cite[17.2]{Garrett}, we know that $\partial X(G,k)$ is a spherical building whose Weyl group $W$ is that of a special vertex of $X(G,k)$.
In particular, there is a spherical Tits system $(H,B,N,S)$ of type $W$ where $B$ denotes the stabilizer in $H$ of a chamber at infinity.
Let $D \subset X(G,k)$ be a sector so that the face $[D]$ of $D$ at infinity is the chamber fixed by $B$.
Let $S$ be a maximal $k$-split torus of $G$ such that $D \subset A_S$.
Then, there is a choice of positive roots $\Phi^+ \subset \Phi(G,S)$ such that $\forall a \in \Phi^+$, we have $U_{a,D} \neq \{1\}$ and $U_{-a,D} = 1$.
Moreover, for any $u \in U_a(k)$, there is a subsector of $D$ fixed by $u$.
Hence $U_a(k)$ fixes the chamber at infinity $[D]$.
By \cite[21.11]{Borel}, we know that the group $U^+$ generated by the $U_a(k)$ is the group of rational points of the unipotent radical of a minimal $k$-parabolic subgroup of $G$.
By construction, we have that $B \supset U^+$.

Let $\xi \in \partial X(G,k)$ be any point of the visual boundary of the building and denote by $P_\xi$ the stabilizer in $H$ of $\xi$.
By \cite[5.6]{Garrett}, we know that $P_\xi$ is a strict parabolic subgroup of the Tits system $(H,B,N,R_1)$ and, in particular, $P_\xi$ contains a conjugate of $B$, therefore of $U^+$.
We also know by \cite[21.15]{Borel} that $G(k)$ admits a Tits system $\left( G(k), P(k), \mathcal{N}_G(S)(k), R_2 \right)$ where $S$ is a maximal $k$-split torus and $P$ is a minimal $k$-parabolic subgroup with Levi factor $\mathcal{Z}_G(S)$.
Up to conjugacy \cite[20.9]{Borel}, we can assume that $U^+$ is the group of rational points of the unipotent radical of $P$.
By \cite[IV.2.7 Remark 1 of Theorem 5]{Bourbaki4-6} applied to $G' = G(k)^+ = G_1$, $H=1$ and $B = P(k)$, we obtain the existence of Tits system $\left(G(k)^+,U^+,N',R_2'\right)$ so that $P_\xi \cap G(k)^+$ is a parabolic subgroup of this Tits system since it contains a conjugate of $U^+$.
Moreover, $P_\xi \cap G(k)^+ \neq G(k)^+$ since $G(k)^+$ does not fix any point of $X(G,k)$.
As a consequence, $P_\xi$ is contained in the rational points of a strict $k$-parabolic subgroup of $G$.
In particular, up to conjugacy, there is a root $a \in \Phi$ such that $U_a(k) \cap P_\xi = 1$ that is not open in $U_a(k) \subset H$.
Hence $P_\xi$ is not an open subgroup of $H$.

By contradiction, assume that $\Omega$ is not bounded.
Then $\Omega$ contains a geodesic ray consisting on points fixed by $U$.
By definition of $\partial X(G,k)$, this ray defines a point $\xi \in \partial X(G,k)$.
Such a point $\xi$ is fixed by $U \cap H$ since $U$ fixes the geodesic ray.
Hence, the stabilizer $P_\xi$ of $\xi$ in $H$ contains $U \cap H$, thus it is open in $H$.
This is a contradiction.
\qed
\end{proof}

\begin{Prop}\label{prop:compact:open:finitely:contained}
Every compact open subgroup of $G(k)$ is contained in finitely many compact (open) subgroups of $G(k)$.
\end{Prop}

\begin{proof}
Consider a compact open subgroup $U \subset G(k)$.
By Lemma \ref{lem:compact:open:fixed:points}, the set $\Omega = X(G)^U$ is non-empty and bounded.
By Remark \ref{rq:compact:contained:in:maximal}, $U$ is contained in a maximal compact subgroup.
Since $X(G,k)$ is locally finite (because $k$ is a local field), by Remark \ref{rq:isolated:maximal:bounded} $U$ is contained in finitely many maximal compact subgroups.
Since $U$ is open, it has finite index in any maximal compact subgroup.
Hence, the set of compact subgroups containing $U$ is finite.\qed
\end{proof}

We now obtain the first step of the main theorem \ref{thm:equivalence:quasi:reductive} by the following:

\begin{Prop}\label{prop:noetherian:semisimple}
Let $k$ be a non-Archimedean local field and $G$ be an almost $k$-simple, semisimple group.
Then $G(k)$ is Noetherian.
\end{Prop}

\begin{proof}
Let $(U_n)_{n \in \NN}$ an increasing sequence of open subgroups of $G(k)$.
Denote by $G(k)^+$ the normal subgroup of $G(k)$ generated by rational points of the unipotent radical of minimal parabolic $k$-subgroups of $G$ \cite[6.2]{BorelTits-HomoAbstraits}.

Let us start with statements on open subgroups of $G(k)$.
Consider an open subgroup $U$ of $G(k)$.

\textbf{Claim:} If $U \cap G(k)^+$ is bounded, then $U$ is compact.

Indeed, assume $U \cap G(k)^+$ is bounded.
Since $U$ is open, it is closed.
The group $G(k)^+$ is closed according to \cite[6.14]{BorelTits-HomoAbstraits}.
Thus, $U \cap G(k)^+$ is compact by \ref{lem:bounded:versus:compact}(2).
By \cite[6.14]{BorelTits-HomoAbstraits}, the quotient group $G(k) / G(k)^+$ is compact.
Hence $U G(k)^+ / G(k)^+$ is a compact open subgroup of $G(k) / G(k)^+$.
The natural bijective continuous homomorphism $U /(U \cap G(k)^+ ) \rightarrow U G(k)^+ /G(k)^+$ is open and hence a homeomorphism, so $U /(U \cap G(k)^+ )$
is compact.
It follows that $U$ is compact.

\textbf{Claim:} If $U$ is not bounded, then $U$ contains $G(k)^+$

Indeed, if $U$ is not bounded, then it is not compact.
Hence $U \cap G(k)^+$ is a non-bounded open subgroup of $G(k)^+$ by the previous claim.
By a theorem of Prasad, attributed to Tits \cite[Theorem (T)]{Prasad-TitsTheorem}, we get $U \cap G(k)^+ = G(k)^+$.
Hence $U \supset G(k)^+$.

Let us now finish the proof by distinguishing two cases.

\textbf{First case:} $U_n$ is bounded (hence compact) for all $n \in \NN$.

By Proposition \ref{prop:compact:open:finitely:contained}, $U_0$ is contained in finitely many compact subgroups.
Hence, the increasing sequence of compact open subgroups $(U_n)_{n \in \NN}$ is eventually constant.

\textbf{Second case:} $U_N$ is not bounded for some $N \in \NN$.

Hence, for all $n \geq N$, the group $U_n$ is not bounded and contains $G(k)^+$.
The open subgroup $U_N / G(k)^+$ of the compact group $G(k) / G(k)^+$ has finite index.
Hence, the sequence $(U_n)_{n \in \NN}$ is eventually constant.\qed
\end{proof}

\subsection{Quasi-reductive groups}

\subsubsection*{The case of a commutative quasi-reductive group}

\begin{Prop}\label{prop:noetherian:commutative}
Let $k$ be a non-Archimedean local field.
If $C$ is a smooth connected commutative quasi-reductive $k$-group,
then $C(k)$ is Noetherian.
\end{Prop}

\begin{proof}
This proof follows the beginning of the proof of \cite[4.1.5]{Conrad-cosetfinite}.

Let $S$ be the maximal $k$-split torus of $C$ (it is unique by $k$-rational conjugacy \cite[C.2.3]{CGP}).
Consider the smooth quotient of algebraic $k$-groups:
\begin{equation*}
1 \longrightarrow S \stackrel{j}{\longrightarrow}
C \stackrel{\pi}{\longrightarrow} C /S \longrightarrow 1
\end{equation*}

\textbf{Claim:} The connected smooth abelian $k$-group $C/S$ does not contains any subgroup isomorphic to $\mathbb{G}_a$ or $\mathbb{G}_m$.

Applying \cite[Exp. XVII 6.1.1(A)(ii)]{SGA3} to the preimage in $C$ of a subgroup isomorphic to $\mathbb{G}_a$ (see \cite[4.1.4]{Conrad-cosetfinite} for a more direct proof), we get a contradiction with quasi-reductiveness of $C$.
Applying \cite[8.14 Cor.]{Borel} to the preimage in $C$ of a subgroup isomorphic to $\mathbb{G}_m$, we get a contradiction with maximality of $S$.

By Lemma \ref{lem:exact:sequence:cohomology}(b) and Hilbert 90 theorem, we get a short exact sequence of topological groups:
\begin{equation*}
1 \longrightarrow S(k) \stackrel{j_k}{\longrightarrow}
C(k) \stackrel{\pi_k}{\longrightarrow} \big(C/S\big)(k) \longrightarrow 1
\end{equation*}
where $\pi_k$ is a surjective open morphism.

By \cite[A.5.7]{Conrad-cosetfinite}, the topological group $\big(C/S\big)(k)$ is compact, hence it is Noetherian by Proposition \ref{prop:noetherian:groups}(2) (In this commutative case, we also have a direct proof considering the smooth quotient of $C/S$ by its maximal $k$-torus, which is anisotropic).
By Proposition \ref{prop:noetherian:groups}(4) and (5), the topological group $S(k) \simeq (k^\times)^n$ (where $n = \mathrm{dim} S$) is Noetherian.
Applying Proposition \ref{prop:noetherian:groups}(3) to $\pi_k$, the topological group $C(k)$ is Noetherian.\qed
\end{proof}

\subsubsection*{The case of a pseudo-reductive group}

Thanks to \cite{CGP}, we have structure theorems on pseudo-reductive groups, well summarized in \cite[§2]{Conrad-cosetfinite}.
In particular, there is a lot of flexibility in the choice of a (generalised) standard presentation, so that we can reduce the question of Noetherianity from pseudo-reductive groups to semisimple groups and commutative quasi-reductive groups.

\begin{Lem}\label{lem:noetherian:weil:restriction}
Let $k$ be a non-Archimedean local field and $k'$ a nonzero finite reduced $k$-algebra, and write $k' = \prod_{i \in I} k'_i$ where $k'_i/k$ are extensions of local fields of finite degree (but possibly non-separable).
Let $G'$ be a smooth connected $k'$-group and denote by $G'_i$ its fiber over the factor field $k'_i$.
Consider the smooth connected $k$-group $G = R_{k'/k}(G')$.
If each fiber $G'_i$ is either an absolutely simple semisimple $k'_i$-group or a basic exotic pseudo-reductive $k'_i$-group, then the topological group $G(k)$ is Noetherian.
\end{Lem}

\begin{proof}
Write $R_{k'/k}(G') = \prod_{i \in I} R_{k'_i/k}(G'_i)$ \cite[A.5.1]{CGP}.
There is a topological isomorphism $R_{k'/k}(G')(k) \simeq \prod_{i\in I}G'_i(k'_i)$.
If each factor $G'_i(k'_i)$ is Noetherian,
then so is $G(k)$ by Proposition \ref{prop:noetherian:groups}(4).

From now on, assume that $k'/k$ is a finite extension of local fields.
It is sufficient to show that $G'(k')$ is Noetherian.

If $G'$ is an absolutely simple semisimple $k'$-group, then by Proposition \ref{prop:noetherian:semisimple} the topological group $G'(k')$ is Noetherian.

Otherwise, $G'$ is a basic exotic pseudo-reductive $k'$-group (see \cite[7.2]{CGP} or \cite[2.3.1]{Conrad-cosetfinite} for a convenient definition). Hence we are in the case of a field with $\mathrm{char}(k') \in \{2,3\}$.
Then, by \cite[7.3.3, 7.3.5]{CGP}, $G(k')$ is topologically isomorphic to $\overline{G}(k')$ where $\overline{G}$ is an absolutely simple semisimple $k'$-group.
Hence, $G(k')$ is Noetherian again by Proposition \ref{prop:noetherian:semisimple}.\qed
\end{proof}

\begin{Prop}\label{prop:noetherian:pseudo:reductive}
Let $k$ be a non-Archimedean local field and $G$ a pseudo-reductive group.
Then $G(k)$ is Noetherian.
\end{Prop}

\begin{proof}
This proof almost follows the proof of \cite[4.1.9]{Conrad-cosetfinite}, based on structure theorem of pseudo-reductive groups over a local field.
Let us recall the main steps of this proof.

If $k$ is any field of characteristic $p \neq 2,3$, then a pseudo-reductive $k$-group is always standard according to \cite[5.1.1]{CGP}.

If $k$ is a local field of characteristic $p \in \{2,3\}$, then we are in the convenient case of a base field $k$ with $[k:k^p] = p$.
Hence, by theorem \cite[10.2.1]{CGP}, $G$ is the direct product $G_1 \times G_2$ of a generalised standard pseudo-reductive $k$-group $G_1$ and a totally non-reduced pseudo-reductive $k$-group $G_2$.
Moreover, the $k$-group $G_2$ is always trivial when $p \neq 2$.

\textbf{First step:} Assume $G_2$ is not trivial (hence $\mathrm{char}(k) = 2$).
By \cite[9.9.4]{CGP}, the topological group $H(k)$, deduced from a basic non-reduced pseudo-simple $k$-group $H$ (see definition \cite[10.1.2]{CGP}) is topologically isomorphic to $\mathrm{Sp}_{2n}(K)$ for some $n$ and an extension of local fields $K/k$.
By Proposition \ref{prop:noetherian:semisimple}, $\mathrm{Sp}_{2n}(K)$ is Noetherian, hence so is $H(k)$.
By \cite[10.1.4]{CGP}, the totally non-reduced $k$-group $G_2$ is isomorphic to a Weil restriction $R_{k'/k}(G_2')$ where $k'$ is a nonzero finite reduced $k$-algebra and fibers of $G_2'$ are basic non-reduced pseudo-simple $k$-groups.
By Lemma \ref{lem:noetherian:weil:restriction}, $G_2(k)$ is Noetherian.

\textbf{Second step:} From now on, we can assume that $G = G_1$ is a generalised standard pseudo-reductive $k$-group, together with a generalised standard presentation $(G',k'/k,T',C)$ and $C' = \mathcal{Z}_{G'}(T')$ where $k'$ is a nonzero finite reduced $k$-algebra, $T'$ is a maximal $k'$-torus of $G'$ and $C$ is a Cartan $k$-subgroup of $G$.
Write $k'=\prod_{i \in I} k'_i$ where $k'_i/k$ are finite extensions of local fields.
By definition of a generalised standard presentation, $G'$ is a $k'$-group whose fibers, denoted by $G'_i$, are absolutely simple simply connected semisimple or basic exotic pseudo-reductive.
Hence, by Lemma \ref{lem:noetherian:weil:restriction}, the topological group $R_{k'/k}(G')(k)$, which is topologically isomorphic to $\prod_{i \in I} G'_i(k'_i)$, is Noetherian.
Moreover, by Propositions \ref{prop:noetherian:commutative} and \ref{prop:noetherian:groups}(4), the topological group $\Big( R_{k'/k}(G') \rtimes C \Big)(k)$ is Noetherian.

\textbf{Third step :} under the above notations $H^1\big(k,R_{k'/k}(C')\big)$ is finite : this is exactly a part of the proof of \cite[4.1.9]{Conrad-cosetfinite}.
By \cite[4.1.6]{Conrad-cosetfinite}, there is a natural group homomorphism $H^1\big(k,R_{k'/k}(C')\big) \simeq \prod_{i \in I} H^1\big(k'_i,C'_i\big)$.
If $G'_i$ is semisimple, the cartan subgroup $C'_i$ is a torus and $H^1\big(k'_i,C'_i\big)$ is finite by \cite[4.1.7]{Conrad-cosetfinite}.
Otherwise $G'_i$ is a basic exotic pseudo-reductive group.
There is a quotient map on an absolutely simple semisimple group $G'_i \rightarrow \overline{G'_i}$ carrying $C'_i$ onto a cartan subgroup (a torus) $\overline{C'_i}$ of $\overline{G'_i}$.
Over a separable closure $(k'_i)_s$ the injective map of rational points $C'_i((k'_i)_s) \rightarrow \overline{C'_i}((k'_i)_s)$ becomes bijective.
By \cite[4.1.6]{Conrad-cosetfinite}, there is an isomorphism $H^1(k'_i,C'_i) \simeq H^1(k'_i,\overline{C'_i})$ and the second one is finite by \cite[4.1.7]{Conrad-cosetfinite} again, since $C'_i((k'_i)_s)$ is Galois-equivariantly identified with $(k'_i)_s$-points of a $k'_i$-torus in such cases.

By definition of a generalised standard presentation \cite[10.1.9]{CGP}, we have a group isomorphism:
$$G \simeq \Big(R_{k'/k}(G') \rtimes C \Big) / R_{k'/k}(C')$$
where $R_{k'/k}(C')$ is realized as a central subgroup of $R_{k'/k}(G') \rtimes C$.

Thus, by \cite[I.5.6 Cor. 2]{SerreCohomologieGaloisienne}, we have an exact sequence of group homomorphisms:
$$1 \rightarrow R_{k'/k}(C')(k) \rightarrow \Big(R_{k'/k}(G') \rtimes C \Big)(k) \stackrel{\pi_k}{\rightarrow} G(k) \stackrel{\delta}{\rightarrow} H^1(k,R_{k'/k}(C'))$$
Thus, as in \cite[4.1.9 (4.1.2)]{Conrad-cosetfinite}, the continuous morphism between topological groups $\pi_k : \Big(R_{k'/k}(G') \rtimes C \Big)(k) \rightarrow G(k)$ is open with a normal image which has finite index since the group $H^1(k,R_{k'/k}(C'))$ is finite.
Hence, by \ref{prop:noetherian:groups}(6) applied to this morphism $\pi_k$, the group $G(k)$ is Noetherian.\qed
\end{proof}

\begin{Rq}
In this proof, the strategy was to reduce the problem from rational points of pseudo-reductive groups to an open subgroup of a quotient of rational points of reductive groups.
There exist pseudo-reductive groups whose rational points are not that of some reductive group.
For instance, consider $k'/k$ a finite extension of local fields of characteristic $p$.
The group $G = \frac{R_{k'/k}(\mathrm{SL}_{p,k'})}{R_{k'/k}(\mu_{p})}$ is pseudo-reductive according to \cite[1.3.4]{CGP}. One can show \cite[1.4.7]{CGP} that $G$ is not a $k$-isogenous quotient of a $k$-group of the form $R_{K/k}(H)$ where $K$ is a nonzero finite reduced $k$-algebra and $H$ is a $K$ group whose fibers over $\mathrm{Spec}(K)$ are connected reductive groups.
\end{Rq}

\subsubsection*{General case}

\begin{Prop}\label{prop:noetherian:quasi:reductive}
Let $k$ be a non-Archimedean local field and $G$ be a quasi-reductive group.
Then $G(k)$ is Noetherian.
\end{Prop}

\begin{proof}
Consider the pseudo-reductive quotient of $G$ :
\begin{equation*}
1 \longrightarrow \mathcal{R}_{u,k}(G) \longrightarrow G \stackrel{\pi}{\longrightarrow} G / \mathcal{R}_{u,k}(G) \longrightarrow 1
\end{equation*}

By Lemma \ref{lem:exact:sequence:cohomology}(b) one has the following exact sequence of topological groups:
\begin{equation*}
1 \longrightarrow \mathcal{R}_{u,k}(G)(k) \longrightarrow G(k) \stackrel{\pi_k}{\longrightarrow} \big( G / \mathcal{R}_{u,k}(G) \big) (k)
\end{equation*}
where the homomorphism $\pi_k$ is open because $\mathcal{R}_{u,k}(G)$ is smooth.

Applying \cite[VI.1]{Oesterle} to the $k$-wound unipotent group $\mathcal{R}_{u,k}(G)$, the topological group $\mathcal{R}_{u,k}(G)(k)$ is compact, hence it is Noetherian by Proposition \ref{prop:noetherian:groups}(2).
Applying Proposition \ref{prop:noetherian:pseudo:reductive} to the pseudo-reductive $k$-group $G / \mathcal{R}_{u,k}(G)$, we get that the topological group $\big( G / \mathcal{R}_{u,k}(G) \big) (k)$ is Noetherian.
Hence, by Proposition \ref{prop:noetherian:groups}(3), the topological group $G(k)$ is Noetherian.\qed
\end{proof}

\subsection{Proof of the equivalence theorem}

Now, there are no extra difficulties to prove Theorem \ref{thm:equivalence:quasi:reductive} giving an equivalence between an algebraic property and topological ones.
We prove successively $(iii) \text{ or } (iv) \Rightarrow (i) \Rightarrow (ii) \Rightarrow (iii) \text{ and } (iv)$.

Let us prove $(iii) \text{ or } (iv) \Rightarrow (i)$.

Because this step focuses on the topological properties of the unipotent radical, let us introduce a purely topological group theoretic definition close to this algebraic group notion.

\begin{Def}\label{def:locally:elliptic}
A locally compact group $G$ is \textbf{locally elliptic} if every compact subset of $G$ is contained in a compact open subgroup of $G$.
\end{Def}

\begin{Lem}\label{lem:locally:elliptic:maximal}
Let $G$ be a totally disconnected locally compact group.
Let $R$ be a closed locally elliptic normal subgroup of $G$.
\begin{enumerate}
\item If $G$ admits a maximal compact subgroup, then $R$ is compact.
\item If $G$ admits a maximal pro-$p$ subgroup and if every compact subgroup of $R$ is pro-$p$, then $R$ is pro-$p$.
\end{enumerate}

\end{Lem}

\begin{proof}
Let $U$ be a compact subgroup of $G$. It is locally elliptic.
Then $R \cdot U$ is a closed subgroup of $G$.
Since $U$ and $R$ are locally elliptic, so is the group $R \cdot U$ \cite[4.D.6 (2)]{CornulierHarpe}.
Let $V$ be a compact relatively open subgroup of $R$.
Since $R \cdot U$ is locally elliptic, there exists a compact subgroup $W$ of $R \cdot U$ containing $U$ and $V$.

(1) If $U$ is a maximal compact subgroup of $G$, we get $U = W$ and therefore $V \subset U$.

(2) Assume that every compact subgroup of $R$ is pro-$p$, therefore $V$ is pro-$p$.
The pro-$p$ group $U$ normalizes the group $W \cap R$, which is pro-$p$ being a compact subgroup of $R$.  Hence, the group $W = (W \cap R) \cdot U$ is a pro-$p$ group (as the image of a semi-direct product of pro-$p$ groups $(W \cap R) \rtimes U$ by the surjective morphism $(W \cap R) \rtimes U \rightarrow (W \cap R) \cdot U \subset G$ induced by multiplication  \cite[2.2.1(e)]{RibesZalesskii} and \cite[1.4 Prop.4(b)]{SerreCohomologieGaloisienne}).
If $U$ is a maximal pro-$p$ subgroup of $G$, we get $U = W$ and therefore $V \subset U$.

Since $R$ is locally elliptic, it is the union of its compact subgroups. We deduce in both cases that $R$ is a closed subgroup of $U$, therefore it is compact.
Moreover $R$ is pro-$p$ when $U$ is pro-$p$.
\qed
\end{proof}

In particular, this will be applied to the locally elliptic radical \cite[4.D.7 (7)]{CornulierHarpe} of a topological group. Now, we go back to algebraic groups.

\begin{Lem} \label{lem:unipotent:open:pro-p:covering}
Let $k$ be a non-Archimedean local field.
If $U$ is a smooth connected affine unipotent $k$-group,
then $U(k)$ is the union of an increasing sequence, indexed by $\ZZ$, of pro-$p$ open subgroups $(U_n)_{n \in \ZZ}$ whose intersection is trivial.

Moreover, when $U$ is not $k$-wound, one can assume that $U_n$ is strictly increasing.
\end{Lem}

\begin{proof}
Denote by $\varpi$ a uniformizer of $\mathcal{O}_k$ and, for all $n \in \ZZ$, denote $\mathfrak{m}^n = \varpi^n \mathcal{O}_k \subset k$.
Denote by $\mathbb{U}_{m}$ the smooth connected unipotent $k$-split $k$-group of upper triangular unipotent matrices.
For $n \in \ZZ$, define
\begin{equation*}
  P_n = \left\{ (x_{i,j})_{1 \leq i,j \leq m}\,,\,
    \begin{array}{cl}
      x_{i,j} = 0 & \text{ if } i > j\\
      x_{i,j} = 1 & \text{ if } i = j\\
      x_{i,j} \in \mathfrak{m}^{n(i-j)} & \text{ if } i < j
    \end{array} 
    \right\}
     \subset \mathbb{U}_m(k)
\end{equation*}

The sequence $(P_n)_{n \in \ZZ}$ is an increasing sequence of groups whose intersection is trivial and union is equal to $\mathbb{U}_{m}(k)$.
For all $n$, the subgroup $P_n$ of $\mathbb{U}_{m}(k)$ is open since it contains the open neighbourhood of identity ${\Big(1 + \mathfrak{m}^{|n|(m-1)} \mathcal{M}_m(k)\Big) \cap \mathbb{U}_m(k)}$.
And it is a pro-$p$-group since every $P_{n+1}$ is a normal subgroup of $P_n$ such that the quotient $P_n / P_{n+1}$ is a $p$-group.

By \cite[15.5(ii)]{Borel}, there is a closed immersion $U \rightarrow \mathbb{U}_m$ for some $m \in \NN$.
Define $U_n = P_n \cap U(k)$.
The sequence of subgroups $(U_n)_{n\in \ZZ}$ is increasing,
the intersection $\displaystyle \bigcap_{n \in \ZZ} U_n$ is trivial
and the union is $\displaystyle \bigcup_{n \in \ZZ} U_n = U(k)$,
because the same holds for $(P_n)$ and $\mathbb{U}_m(k)$.
Every $U_n$ is a pro-$p$ subgroup of $U(k)$ because ${U(k) \subset \mathbb{U}_m(k)}$ is closed,
and it is an open subgroup of $U(k)$ because $P_n$ is open in $\mathbb{U}_m(k)$.

Now, assume that $U$ is not $k$-wound.
Since $U(k)$ is not compact by \cite[VI.1]{Oesterle}, every $U_n$ is distinct from $U(k)$.
Moreover, $U_n$ is never trivial because it is open in $U(k)$.
Hence, one can extract a strictly increasing sequence $\big(U'_{\varpi(n)}\big)_{n \in \ZZ}$ with the same properties as before.\qed
\end{proof}

In particular, this lemma tells us that unipotent groups over local fields are locally elliptic. More precisely:

\begin{Prop} \label{prop:maximal:compact:existence:implies:quasireductive}
Let $k$ be a non-Archimedean local field of residual characteristic $p$ and $G$ be a smooth connected affine $k$-group.
\begin{enumerate}
\item The rational points $\mathcal{R}_{u,k}(G^0)(k)$ of the unipotent $k$-radical of $G$ is a closed locally elliptic normal subgroup of $G(k)$ of which any compact subgroup is pro-$p$.
\item Assume that the topological group $G(k)$ contains either a maximal pro-$p$ subgroup, or a maximal compact subgroup. Then, $G^0$ is a quasi-reductive $k$-group.
\end{enumerate}
\end{Prop}

\begin{proof}
(1) Denote by $U = \mathcal{R}_{u,k}(G^0)$ the unipotent $k$-radical of $G$.
By Lemma \ref{lem:unipotent:open:pro-p:covering} and Lemma \cite[2.3]{Caprace-amenable}, we get that every compact subset of $U(k)$ is contained in some pro-$p$ open subgroup. In particular, $U(k)$ is locally elliptic and any compact subgroup of $U(k)$ is pro-$p$ as closed subgroup of a pro-$p$ group.

(2) By Lemma \ref{lem:locally:elliptic:maximal}, we get that $U(k)$ is compact.
Hence, by \cite[VI.1]{Oesterle}, the topological group $U$ is $k$-wound.
In other words, the connected algebraic group $G^\circ$ is quasi-reductive.
\qed
\end{proof}

Let us prove $(i) \Rightarrow (ii)$.

\begin{Prop} \label{prop:noetherian:general:case}
Let $k$ be a non-Archimedean local field and $G$ a smooth affine $k$-group.
If $G^0$ is a quasi-reductive $k$-group, then $G(k)$ is Noetherian.
\end{Prop}

\begin{proof}
The identity component $G^0$ of $G$ is a smooth normal $k$-subgroup of $G$ \cite[II.§5 1.1 and 2.1]{DemazureGabriel}, and the quotient $F=G/G^0$ is a (smooth) finite $k$-group \cite[II.§5 1.10]{DemazureGabriel}.

By Lemma \ref{lem:exact:sequence:cohomology}(b), we have an exact sequence of topological groups
\begin{equation*}
1 \rightarrow G^0(k) \rightarrow G(k) \stackrel{\pi_k}{\rightarrow} F(k)
\end{equation*}
where $\pi_k$ is an open morphism.

By Proposition \ref{prop:noetherian:quasi:reductive}, the topological group $G^0(k)$ is Noetherian and $F(k)$ is Noetherian because it is finite.
As a consequence, by Proposition \ref{prop:noetherian:groups}(3), the topological group $G(k)$ is Noetherian.\qed
\end{proof}

To conclude, let us finish the proof by showing that $(ii) \Rightarrow (iii) \text{ and } (iv)$.

\begin{Prop}\label{prop:noetherian:implies:maximal}
Let $G$ be a Noetherian totally disconnected locally compact group.
It admits a maximal compact subgroup.
Moreover, if $G$ admits an open pro-$p$ subgroup, then it admits a maximal pro-$p$ subgroup.
\end{Prop}

\begin{proof}
By contradiction, assume than $G$ does not contains a maximal pro-$p$ (resp. compact) subgroup.

By induction, it is possible to define a strictly increasing sequence of pro-$p$ (resp. compact) open subgroups.
Basis of the induction is given by the existence of an open pro-$p$ (resp. compact) subgroup.
Induction step: since $G$ does not admit a maximal pro-$p$ (resp. compact) subgroup, given a pro-$p$ (resp. compact) open subgroup $U_n$, there exists a pro-$p$ (resp. compact) subgroup $U_{n+1}$ containing $U_n$ strictly.
The group $U_{n+1}$ is open since it contains $U_n$.

Such a sequence cannot exist since $G$ is Noetherian: there is a contradiction.\qed
\end{proof}

In particular, by applying Lemma \ref{lem:existence:open:pro-p} and Proposition \ref{prop:noetherian:implies:maximal} to a smooth affine algebraic group $G$ over a non-Archimedean local field $k$ whose the group of rational points $G(k)$ is Noetherian, we get the implications $(ii) \Rightarrow (iii) \text{ and } (iv)$ of Theorem \ref{thm:equivalence:quasi:reductive}.

Let us now prove the second part of Theorem \ref{thm:equivalence:quasi:reductive}.

\begin{Lem}\label{lem:link:with:openness}
Let $G$ be a totally disconnected locally compact group.
Every compact subgroup of $G$ is contained in an open compact subgroup of $G$.
Moreover, if $G$ contains a open pro-$p$ subgroup $U$, then every pro-$p$ subgroup of $G$ is contained in some open pro-$p$ subgroup of $G$.
\end{Lem}

\begin{proof}
Let $P$ be a compact subgroup of $G$ and $U$ be an open compact subgroup of $G$.
The index $[P : U \cap P]$ is finite since $P$ is compact and $U\cap P$ is open in $P$.
Hence, the set $\{x^{-1} U x \,,\,x \in P \}$ is finite.
Define $U_0 = \bigcap_{x \in P} x^{-1} U x$.
It is an open compact subgroup of $G$ normalised by $P$, and it is pro-$p$ when $U$ is pro-$p$.
Hence the group $P_0 = P \cdot U_0$ is an open subgroup of $G$.
It is compact as the image of $P \times U_0$ by the continuous multiplication map $G \times G \rightarrow G$.
When, moreover, $P$ and $U$ are pro-$p$, the group $P_0$ is pro-$p$ as the image of the pro-$p$ group $P \ltimes U_0$ by the surjective multiplication homomorphism $P \ltimes U_0 \rightarrow P \cdot U_0$.\qed
\end{proof}

\begin{proof}[Proof of second part of Theorem \ref{thm:equivalence:quasi:reductive}]
By Lemma \ref{lem:existence:open:pro-p} applied to $G(k)$, we can apply Lemma \ref{lem:link:with:openness} to $G(k)$.
Using the same construction by induction as in proof of \ref{prop:noetherian:implies:maximal}, statements (1) and (2) are a direct result from Noetherianity and Lemma \ref{lem:link:with:openness}.\qed
\end{proof}

\subsection{Topological structure of pseudo-reductive groups over a local field}

In order to prove the existence of a maximal compact subgroup of a pseudo-reductive group over a local field, the standard presentation of these latter was used. This presentation actually provides a finer way to describe the structure of pseudo-reductive groups.
Through standard presentations, a pseudo-reductive group can be described as a quotient of some semi-direct product by a normal subgroup.
We begin by describing the structure of the rational points of such a subgroup, as follows:

\begin{Lem}\label{lem:product:generalized:standard}
If $G$ is a generalized standard pseudo-reductive $k$-group arising from a generalised standard presentation $(G',k'/k,T',C)$.
Then $R_{k'/k}(G')(k)$ is topologically isomorphic to a direct product of groups that are each isomorphic to the rational points of a simply connected absolutely simple algebraic group over a local field of the same characteristic and residue characteristic as $k$.
\end{Lem}

\begin{proof}
Write $k'=\prod_{i \in I} k'_i$ where $k'_i/k$ are finite extensions of local fields and denote by $G'_i$ the fibers of $G' \to \operatorname{Spec}(k')$ that are simply connected groups by definition \cite[10.1.9]{CGP}.
The group $R_{k'/k}(G')$ is topologically isomorphic to $\prod_i G'_i(k'_i)$.
If $G'_i$ is a simply connected basic exotic pseudo-reductive $k'_i$-group, then by \cite[7.3.3 (2)]{CGP} there is a topological isomorphism $f_i$ from $G'_i(k'_i)$ to a group $\overline{G'_i}(k'_i)$ where $\overline{G'_i}$ is simply connected by \cite[7.1.5]{CGP} and an absolutely simple $k'_i$-group according to \cite[7.3.5]{CGP}.
If $G'_i$ is absolutely pseudo-simple in characteristic $\operatorname{char}(k'_i) = 2$ with a non-reduced root system, then Proposition \cite[9.9.4(2)]{CGP} provides a field $K_i$ and a topological isomorphism $f_i : G'_i(k'_i) \to \mathrm{Sp}_{2n}(K_i)$.
Otherwise, it means that $G'_i$ is simply connected and absolutely simple \cite[10.2.1]{ConradPrasad} and one take $f_i$ to be the identity map.
Thus, one gets the suitable topological isomorphism $\prod_i f_i$.
\qed
\end{proof}

Thanks to the generalized standard presentation (with suitable assumptions in characteristic $2$, satisfied by the local fields) it is possible to identify a pseudo-reductive group with a quotient of a suitable semi-direct product by a central subgroup whose $H^1$ is finite.
Of course, one cannot identify the rational points of this quotient to the quotient of the rational points of these algebraic groups.
Nevertheless, we will be able to use the topological properties of the restriction to the rational points of the quotient map of algebraic groups in order to define quotient maps of topological groups by some well-chosen subgroups.
In this case, we do a quotient by the locally elliptic radical which is a topologically characteristic subgroup.
The established existence of maximal compact subgroups of the rational points of a pseudo-reductive group will allow us to reduce questions of local ellipticity to questions of compactness of groups.

\begin{Prop}\label{prop:structure:pseudo:reductive}
Let $G$ be a connected pseudo-reductive group over a non-Archimedean local field $k$.
The group of rational points $G(k)$ admits a chain of closed normal subgroups $1 \leqslant Q \leqslant S \leqslant G(k)$ such that
$Q$ is the locally elliptic radical of $G(k)$, therefore a compact group,
the quotient $S/Q$ is the internal direct product of finitely many non-compact, topologically simple, compactly generated locally compact groups
that are each isomorphic to the quotient of a simply connected isotropic simple algebraic group over a local field (of the same characteristic and residue characteristic as $k$) by its center,
and the quotient $G(k)/S$ is compactly generated and virtually abelian.
\end{Prop}

We prove this proposition by the following steps:
firstly, we study the rational points of the factor $R_{k'/k}(G')$ appearing in the standard presentation and its locally elliptic radical ;
secondly, we compare this locally elliptic radical to that of the semi-direct product appearing in the standard presentation ;
finally, we pass to the quotient in the standard presentation.

\begin{proof}
Given a locally compact group $H$, denote by $\operatorname{Rad}(H)$ its locally elliptic radical, that is a topologically characteristic subgroup, therefore a normal subgroup of $H$.

By Theorem \ref{thm:equivalence:quasi:reductive}, the group $G(k)$ has a maximal compact subgroup. Thus, its locally elliptic radical $Q = \operatorname{Rad}\big( G(k) \big)$ is compact by Lemma \ref{lem:locally:elliptic:maximal}.

Assume firstly that $G=G_1$ is a generalized standard pseudo-reductive group arising from a generalised standard presentation $(G',k'/k,T',C)$.
Recall that there is a subgroup $C'=\mathcal{Z}_{G'}(T')$ of $G'$
such that $R_{k'/k}(C')$ is a Cartan subgroup of $R_{k'/k}(G')$,
together with two $k$-group homomorphisms $(\varphi,\psi)$ providing
a factorization $R_{k'/k}(C') \stackrel{\varphi}{\longrightarrow} C \stackrel{\psi}{\longrightarrow} Z_{R_{k'/k}(G'), R_{k'/k}(C')}$
where $Z_{R_{k'/k}(G'), R_{k'/k}(C')}$ is some subgroup of the group
of $k$-automorphisms of $R_{k'/k}(G')$
that restricts to the identity of $R_{k'/k}(C')$.
This defines the semi-direct product $R_{k'/k}(G') \rtimes_\psi C$ and the commutative subgroup $R_{k'/k}(C')$ is realized in this semi-direct product as a central subgroup through the anti-diagonal embedding $(\iota,\varphi) : R_{k'/k}(C') \to R_{k'/k}(G') \rtimes_\psi C$ mapping $c' \mapsto \left( (c')^{-1}, \varphi(c') \right)$.

According to Lemma \ref{lem:product:generalized:standard}, the group $G'(k')$ is topologically isomorphic to a direct product $X = \prod_{i \in I} \widetilde{G_i}(k_i)$ where $k_i/k$ is a finite extension of local fields and $\widetilde{G_i}$ is a simply connected $k_i$-simple group.
For $i \in I$, denote by $X_i = \widetilde{G_i}(k_i)$ and by $\widetilde{G_i}(k_i)^+$ the subgroup of $\widetilde{G_i}(k_i)$ generated by all the unipotent elements of $\widetilde{G_i}(k_i)$ which are contained in the unipotent radical of some $k_i$-parabolic subgroup of $\widetilde{G_i}$.
Denote by $J$ the subset of elements $j \in I$ such that $\widetilde{G_j}$ is $k_j$-isotropic.
Since $\widetilde{G_j}$ is a $k_j$-isotropic reductive group,
its rational points $X_j$ is not a solvable group.
On the one hand, for any $j \in J$, since the extension $k_j/k$ is finite and $\widetilde{G_j}$ is simply connected, the Kneser-Tits problem \cite[§2]{PrasadRaghunathan-KneserTits} for non-Archimedean local fields says that $\widetilde{G_j}(k_j) = \widetilde{G_j}(k_j)^+$.
Moreover, by \cite[Main Theorem]{Tits-AlgebraicAbstract} the quotient of $\widetilde{G_j}(k_j)^+$ by its (abstract) center $\mathcal{Z}\big( \widetilde{G_j}(k_j) \big)$ is a non-solvable simple group.
It is non-compact by \cite[Theorem BTR]{Prasad-TitsTheorem} because $\mathcal{Z}\big( \widetilde{G_j}(k_j) \big)$ is finite.
For $j \in J$, denote by $\widehat{X_j} = \mathcal{Z}\big( G_j(k_j) \big)$.
On the opposite, for $j' \in I \setminus J$, the group $\widetilde{G_{j'}}$ is $k_{j'}$-anisotropic. Hence, $\widetilde{G_{j'}}(k_{j'})$ is compact by \cite[Theorem BTR]{Prasad-TitsTheorem}. Denote $\widehat{X_{j'}} = \widetilde{G_{j'}}(k_{j'})$.

We have $\operatorname{Rad}(X) = \operatorname{Rad}\big( \prod_{i\in I} \widetilde{G_i}(k_i) \big) = \prod_{i \in I} \widehat{X_i}$.
Indeed, the normal subgroup $\prod_{i \in I} \widehat{X_i}$ of $X$ is compact and, therefore, closed and locally elliptic so that it is contained in $\operatorname{Rad}(X)$.
Conversely, for any $i \in I$, the quotient $X_i / \widehat{X_i}$ is either a trivial group or a non-compact simple group, so that the projection of the normal compact subgroup $\operatorname{Rad}\big( X \big)$ of $X$ on the $j$-th factor is compact, therefore contained in $\widehat{X_j}$.

Now, the semi-direct product $G'(k') \rtimes_{\psi} C(k)$ arising from the standard presentation of $G$ is topologically isomorphic to a semi-direct product denoted by $Y = X\rtimes_{\psi_k} C(k)$. Since $R_{k'/k}(G') \rtimes_\psi C$ is a pseudo-reductive $k$-group, the group $Y$ has a maximal compact subgroup by Theorem \ref{thm:equivalence:quasi:reductive}, hence the locally elliptic radical $\operatorname{Rad}(Y)$ is compact by Lemma \ref{lem:locally:elliptic:maximal}.
Thus, since $X \cap \operatorname{Rad}(Y)$ is a normal compact subgroup of $X$, it is locally elliptic and therefore contained in $\operatorname{Rad}(X)$.
Conversely, since $\operatorname{Rad}(X)$ is a topologically characteristic subgroup of $X$ that is a closed normal subgroup of $Y$, we deduce that $\operatorname{Rad}(X)$ is a locally elliptic closed normal subgroup of $Y$.
Thus $\operatorname{Rad}(X) \subset \operatorname{Rad}(Y)$ and therefore $\operatorname{Rad}(X) = X \cap \operatorname{Rad}(Y)$.

Denote by $\mathcal{D}(H)$ the (abstract) derived subgroup of an abstract or topological group $H$.
Let $q : Y \to Y / \operatorname{Rad}(Y)$ be the quotient map.
Then $Z = q(X)$ is isomorphic to $X / X \cap \operatorname{Rad}(Y)$.
Since we have seen that $X / X \cap \operatorname{Rad}(Y) = X / \operatorname{Rad}(X) = \prod_{j \in J} X_j / \mathcal{Z}(X_j)$ is the direct product of non-commutative simple groups, it follows that $Z$ is a perfect group.
Since $Y / X$ is commutative, the derived group $\mathcal{D}(q(Y)) = \mathcal{D}\big(Y / \operatorname{Rad}(Y)\big)$ is a subgroup of $Z=q(X)$.
Because $Z$ is perfect, we also have that $Z = \mathcal{D}(q(X)) \subset \mathcal{D}(q(Y))$ so that $\mathcal{D}\big(Y / \operatorname{Rad}(Y)\big)$ is isomorphic to $\prod_{j \in J} X_j / \mathcal{Z}(X_j)$.

Now, consider the exact sequence deduced from that of a generalised standard presentation, given by \cite[I.5.6 Cor. 2]{SerreCohomologieGaloisienne}, of group homomorphisms:
$$1 \rightarrow R_{k'/k}(C')(k) \rightarrow \Big(R_{k'/k}(G') \rtimes C \Big)(k) \stackrel{\pi_k}{\rightarrow} G(k) \stackrel{\delta}{\rightarrow} H^1(k,R_{k'/k}(C'))$$
and identify $Y$ with $\Big(R_{k'/k}(G') \rtimes C \Big)(k)$ up to a topological isomorphism.
Since the locally elliptic radical is topologically characteristic and since $\pi_k$ has open, therefore closed, normal image by Lemma \ref{lem:exact:sequence:cohomology}, it follows that $\pi_k(\operatorname{Rad}(Y))$ is a topological characteristic subgroup of the image of $\pi_k$, therefore a closed normal locally elliptic subgroup of $G(k)$.
Thus $\pi_k(\operatorname{Rad}(Y))$ is contained in $Q = \operatorname{Rad}(G(k))$.
Hence $\pi_k$ induces a continuous homomorphism with finite index open normal image
$$\overline{\pi_k} : Y / \operatorname{Rad}(Y) \to G(k) / Q$$
so that we have the following commutative diagram of continuous group homomorphisms:
\begin{equation*}
	\xymatrix{
		X \ar@{^{(}->}[r] \ar@{->>}[d] \ar@{-->}[rrd]^{\phi} & Y \ar[r]^{\pi_k} \ar@{->>}[d] & G(k) \ar@{->>}[d] \\
		X/\operatorname{Rad}(X) \ar@{^{(}->}[r] & Y / \operatorname{Rad}(Y) \ar[r]^{\overline{\pi_k}} & G(k) / Q
	}
\end{equation*}

According to \cite[Cor 5.2]{BaderGelander}, for $i \in I$, every continuous homomorphism $\phi$ of the group $X_i = \widetilde{G_i}(k_i)$ to a locally compact group has closed image and it induces a homeomorphism $X_i / \ker(\phi) \to \phi(X_i)$.
We apply it to $\pi_k$, $\overline{\pi_k}$ and the continuous homomorphism $\phi$ obtained by composition:
$$X \twoheadrightarrow X / \operatorname{Rad}(X) \hookrightarrow Y / \operatorname{Rad}(Y) \stackrel{\overline{\pi_k}}{\longrightarrow} G(k) / Q.$$

For $j \in J$, the image of $X_j \subset Y$ in $G(k)$ by $\pi_k$ is non-trivial because $X_j$ is non-solvable and the kernel of $\pi_k$ is a central subgroup of $Y$.
Thus, the closed subgroup $\pi_k(X_j)$ of $G(k)$ is homeomorphic to $X_j / X_j \cap \ker \pi_k$ with $X_j \cap \ker \pi_k \subset \mathcal{Z}(X_j)$ that is a closed non-compact subgroup of $G(k)$, therefore not contained in $Q$.
Hence the simple group $\phi(X_j) = \overline{\pi_k}(X_j / \mathcal{Z}(X_j))$ is non-trivial, thus topologically isomorphic to $X_j / \mathcal{Z}(X_j)$.

{\bf Claim:} the homomorphism $\overline{\pi_k}$ induces a topological isomorphism $X / \operatorname{Rad}(X) \stackrel{\simeq}{\longrightarrow} \phi(X)$ and $\phi(X)$ is the internal direct product of simple groups $\phi(X_j)$ for $j \in J$.\\
We have shown that $\overline{\pi_k}$ is a continuous open homomorphism and that $\phi(X_j) = \overline{\pi_k}\left(X_j / X_j \cap \operatorname{Rad}(X) \right)$ is topologically isomorphic to $X_j/ \mathcal{Z}(X_j)$.
Since $\phi(X)= \overline{\pi_k}\left( X/ \operatorname{Rad}(X) \right)$, generated by the $\phi(X_j)$ is a closed subgroup of $G(k)/Q$, it suffices to prove that $\ker \overline{\pi_k} \cap X / \operatorname{Rad}(X)$ is trivial so that $\overline{\pi_k}$ realizes a topological isomorphism between $X/ \operatorname{Rad}(X)$ and $\phi(X)$.
Let $x \in X$ such that the coset $[x] \in X/\operatorname{Rad}(X)$ of $x$ with respect to $\operatorname{Rad}(X)$ is in the kernel of $\overline{\pi_k}$.
Then $\overline{\pi_k}([x]) = [\pi_k(x)] = Q$ so that $x \in P = \pi_k^{-1}(Q) \cap X$ that is a closed normal subgroup of $X$.
If we show that $P$ is compact therefore locally elliptic, then it is a subgroup of $\operatorname{Rad}(X)$ and we are done.

Consider the restriction ${\pi_k}|_X$ of the homomorphism $\pi_k$ to the subgroup $X \rtimes \{1\}$ so that $P = {{\pi_k}|_X}^{-1}(Q)$.
We prove that ${\pi_k}|_X : X \rtimes \{1\} \to \pi_k(X)$ is a proper map.
By \cite[3.3]{BaderGelander}, we know that $X$, being a direct product of rational points of connected semisimple groups over local fields, is quasi-semisimple (see definition \cite[3.2]{BaderGelander}).
Thus, by \cite[5.1]{BaderGelander}, the group homomorphism ${\pi_k}|_X$ induces a topological isomorphism $X / X \cap \ker \pi_k \to \pi_k(X)$.
To conclude that ${\pi_k}|_X$ is proper, it suffices to show that $\left( X \rtimes \{1\} \right) \cap \ker \pi_k$ is compact.
The exact sequence of the standard presentation gives us that $\ker \pi$ is the image of the anti-diagonal embedding of $R_{k'/k}(C')$ into $R_{k'/k}(G') \rtimes_\psi C$ through $(\iota,\varphi)$.
Then
\begin{align*}
\ker {\pi_k}|_X & = \left(\ker \pi \right) (k) \cap \left( X \rtimes \{1\} \right)\\
& = \left\{ c \in R_{k'/k}(C')(k),\ \left(c^{-1},\varphi(c)\right) \in X\rtimes_\psi \{1\} \right\}\\
& = \ker \varphi (k)
\end{align*}
But, according to \cite[10.1.10]{CGP}, we know that $\ker \varphi$ is a central subgroup of $R_{k'/k}(G')$.
Thus $\ker {\pi_k}|_X$ is finite.

Let $S$ be the subgroup of $G(k)$ containing $Q$ such that $S/Q = \phi(X)$.
Being an internal direct product of closed subgroups, $S/Q$ is also a closed subgroup of $G(k)/Q$.
Since $\phi(X) = \overline{\pi_k}\big( X / \operatorname{Rad}(X) \big) = \mathcal{D}\Big( \overline{\pi_k}\big( X / \operatorname{Rad}(X) \big) \Big)$ is a characteristic subgroup of  $\overline{\pi_k}\big( Y / \operatorname{Rad}(Y)\big)$ that is an open normal subgroup of $G(k)/Q$, it follows that $S/Q$ is a closed normal subgroup of $G(k)/Q$ topologically isomorphic to an internal direct product of simple groups of the form $\prod_{j \in J} X_j / \mathcal{Z}(X_j)$.

Since $\big(G(k)/Q\big) \Big/ \operatorname{im}(\overline{\pi_k})$ is finite and $\operatorname{im}(\overline{\pi_k}) \Big/ \big( S / Q \big)$ is abelian by construction, it follows that $G(k)/S$ is virtually abelian.

If $k$ is any field of characteristic $p \neq 2,3$, then a pseudo-reductive $k$-group is always standard according to \cite[5.1.1]{CGP}.
Because $k$ is a local field of characteristic $p \in \{2,3\}$, we are in the case of a base field $k$ with $[k:k^p] = p$.
Hence, by theorem \cite[10.2.1]{CGP}, $G$ is the direct product $G_1 \times G_2$ of a generalised standard pseudo-reductive $k$-group $G_1$ and a totally non-reduced pseudo-reductive $k$-group $G_2$.
Moreover, the $k$-group $G_2$ is always trivial when $p \neq 2$.

From now on, assume that $G_2$ is not trivial (hence $\operatorname{char}(k) = 2$) and let $S_1 = S$ and $Q_1=Q$.
By \cite[9.9.4]{CGP}, the topological group $H(k)$, deduced from a basic non-reduced pseudo-simple $k$-group $H$ (see definition \cite[10.1.2]{CGP}) is topologically isomorphic to $\mathrm{Sp}_{2n}(K)$ for some $n$ and an extension of local fields $K/k$.
By \cite[10.1.4]{CGP}, the totally non-reduced $k$-group $G_2$ is isomorphic to a Weil restriction $R_{k'/k}(G_2')$ where $k'$ is a nonzero finite reduced $k$-algebra and fibers of $G_2'$ are basic non-reduced pseudo-simple $k$-groups.
Thus set $S_2 = G_2(k)$ and $Q_2=\operatorname{Rad}(G_2(k))$.
Then $S_1 \times S_2$ and $Q_1 \times Q_2$ satisfy the conditions.
\qed
\end{proof}

From this, we deduce an analogous statement for any open subgroup of rational points of a connected quasi-reductive group (see Theorem \ref{thm:structure:quasi:reductive}). The only difference is that $Q$ might a priori differ from the locally elliptic radical of $G(k)$.

\begin{Rq}
Now, assume that $G$ is a connected quasi-reductive $k$-group and consider the following exact sequence
\begin{equation*}
1 \longrightarrow U \longrightarrow G \stackrel{\pi}{\longrightarrow} G_{\mathrm{p-red}} \longrightarrow 1
\end{equation*}
where $U = \mathcal{R}_{u,k}(G)$ is the unipotent $k$-radical of $G$ and $G_{\mathrm{p-red}} = G / U$ is pseudo-reductive.
If this sequence is split, then one can find a closed $k$-subgroup $L$ of $G$ isomorphic to $G_{\mathrm{p-red}}$ such that $G$ is isomorphic to $U \rtimes L$.
Then, because $U(k)$ is compact by \cite[VI.1]{Oesterle}, one can easily check that $\operatorname{Rad}(G(k)) = U(k) \rtimes \operatorname{Rad}(L(k))$.
Let $Q,S$ be the closed normal subgroups of $L(k)$ as in Proposition \ref{prop:structure:pseudo:reductive}.
Then, by considering $\mathbf{Q} = \operatorname{Rad}(G(k)) = U(k) \rtimes Q$ and $\mathbf{S} = U(k) \rtimes S$, one has $G(k) / \mathbf{S} \simeq L(k) / S$ and $\mathbf{S} / \mathbf{Q} \simeq S / Q$ so that the statement the Theorem extends immediately to such a group $G$.

In general, there is no reason for the existence of such a subgroup $L$. For instance, if the pseudo-reductive quotient is in fact reductive, there are known example in positive characteristic of groups without Levi $k$-subgroups (see \cite[A.6]{CGP}.
In characteristic $0$, one know that there are Levi factors, but in that case $U = 1$ because quasi-reductive groups are in fact reductive.
\end{Rq}

Using the pseudo-reductive quotient, one can provide a generalization to quasi-reductive groups as follows:

\begin{proof}[Proof of Theorem \ref{thm:structure:quasi:reductive}]
The above exact sequence providing the pseudo-reductive quotient of $G$ induces, by Lemma \ref{lem:exact:sequence:cohomology}(b), the following exact sequence of topological groups:
\begin{equation*}
1 \longrightarrow U(k) \longrightarrow G(k) \stackrel{\pi_k}{\longrightarrow} G_{\mathrm{p-red}}(k)
\end{equation*}
where the homomorphism $\pi_k$ is open because $U$ is smooth.
Let $V$ be any open subgroup of $G(k)$.
Let $H$ be the image of $V$ through $\pi_k$ in $G_{\mathrm{p-red}}(k)$ that is an open, therefore closed, subgroup of $G_{\mathrm{p-red}}(k)$.
Because the homomorphism $\pi_k$ is continuous and open, it induces a topological isomorphism $\overline{\pi_k} : V / V \cap U(k) \to H$.

Let $S, Q$ be closed normal subgroup of $G_{\mathrm{p-red}}(k)$ given by Proposition \ref{prop:structure:pseudo:reductive}.
We define the closed subgroup $S_H = S \cap H$.
Since $S$ is a normal subgroup of $G_{\mathrm{p-red}}(k)$ and $H$ is an open subgroup of $G_{\mathrm{p-red}}(k)$, we have an isomorphism of locally compact groups $H S / S \simeq H / S_H$.
Since $G_{\mathrm{p-red}}(k) / S$ is compactly generated and virtually abelian, so is its open subgroup $H S/S$, and hence also is the quotient group $H / S_H$.

Let $\pi_Q : S \to S / Q$ be the quotient morphism.
We claim that $\pi_Q(S_H) \simeq (S\cap H) / (Q \cap H)$ has a compact normal subgroup, such that the corresponding quotient group is the internal direct product of finitely many non-compact, topologically simple, compactly generated locally compact groups that are each isomorphic to the quotient of rational points of some simply connected isotropic simple algebraic group over a local field by its center.
Indeed, we know by Proposition \ref{prop:structure:pseudo:reductive} that the quotient $S/Q$ is an internal direct product of simple groups $T_i = \widetilde{G_i}(k_i) / \mathcal{Z}\left(\widetilde{G_i}(k_i)\right)$ for $1 \leqslant i \leqslant \ell$ of the required form.
Since $S_H$ is an open subgroup of $S$, we have that $\pi_Q(S_H)$ is open in $S/Q = T_1 \times \dots \times T_\ell$.
In particular $\pi_Q(S_H)$ contains a subgroup of the form $U_1 \times \dots \times U_\ell$ where $U_i$ is an open subgroup of $T_i$.
For any $1 \leqslant i \leqslant \ell$, because $\widetilde{G_i}$ is simply connected and $k_i$ is a non-Archimedean local field, the Kneser-Tits problem \cite[§2]{PrasadRaghunathan-KneserTits} says that $\widetilde{G_i}(k_i) = \widetilde{G_i}(k_i)^+$.
By a theorem of Prasad (that the latter attributes to Tits) \cite[Theorem (T)]{Prasad-TitsTheorem}, we know that every proper open subgroup of $T_i$ is compact.
Let $\operatorname{pr}_i : S/Q \to T_i$ be the natural projection morphism on the $i$-th factor $T_i$ of $S/Q$ and $T'_i = \operatorname{pr}_i\left( \pi_Q(S_H) \right)$.
Since $T'_i$ is open in $T_i$, we have that either $T'_i$ is compact or $T'_i = T_i$.
Let $I \subset \{ 1,\dots, \ell\}$ be the set of those indices $i$ such that $T'_i$ is compact, and let $\overline{I}$ be its complement.
Let $T_I = \left( \prod_{i \in I} T_i \right)$ and $T_{\overline{I}} = \left( \prod_{i \in \overline{I}} T_i \right)$ so that $S/Q = T_I \times T_{\overline{I}}$.
Then we have $\prod_{i \in I} U_i \subset T_I \cap \pi_Q(S_H) \subset \prod_{i \in I} T'_i$.
It follows that the intersection $T_I \cap \pi_Q(S_H)$ is a compact normal subgroup of $\pi_Q(S_H)$.
On the other hand, for $i \in \overline{I}$, we know that $T_i \cap \pi_Q(S_H)$ is open in $T_i$, and is a normal subgroup of $\pi_Q(S_H)$.
By applying the canonical projection $\operatorname{pr_i}$, we deduce that $T_i \cap \pi_Q(S_H)$ is an open normal subgroup of $T'_i$.
By definition of $\overline{I}$, we have $T'_i = T_i$ and, because $T_i$ is topologically simple, we deduce that $T_i \cap \pi_Q(S_H) = T_i$.
It follows that $T_{\overline{I}} \subset \pi_Q(S_H)$.

Finally, we define $Q_H$ as the inverse image in $H$ of $\pi_Q(S_H) \cap T_I$ through $\pi_Q$.
Since $Q$ is compact, the quotient morphism $\pi_Q$ is proper.
Since $\pi_Q(S_H) \cap T_I$ is a compact normal subgroup of $\pi_Q(S_H)$, we deduce that $Q_H$ is a compact normal subgroup of $H$.

We claim that $S_H / Q_H$ is isomorphic to $T_{\overline{I}}$.
Indeed, let $\operatorname{pr}_{\overline{I}}$ be the natural projection map $S/Q \to T_{\overline{I}}$ and $f = \operatorname{pr}_{\overline{I}} \circ \pi_Q : S_H \to T_{\overline{I}}$.
Then
\begin{align*}
\ker f &= \pi_Q^{-1}\left( \ker \operatorname{pr}_{\overline{I}} \right) \cap H
=\pi_Q^{-1}\left(T_I\right) \cap H&\\
&=\pi_Q^{-1}\left(  T_I \cap \pi_Q\left(S_H\right)\right) \cap H & \text{ since } T_I \subset \pi_Q\left(S_H\right)\\
&=Q_H & \text{ by definition.}
\end{align*}
Thus the surjective map $f$ induces an isomorphism $S_H / Q_H \simeq T_{\overline{I}}$.

Define $ Q_{V} = V\, \cap\, \pi_k^{-1}\left(Q_H\right)$ and $S_{V} = V \,\cap\, \pi_k^{-1}\left(S_H\right)$ that are closed normal subgroup of $V$.
Because $U(k)$ is compact according to \cite[VI.1]{Oesterle}, and so is $U(k) \cap V$, the topological homomorphism $\pi_k : V \to H$ is proper.
Thus $Q_{V}$ is compact.
By construction, we also have the topological isomorphisms $S_{V} / Q_{V} \simeq S_H / Q_H$ and $V / S_{V} \simeq H / S_H$ that are of the required form.
\qed
\end{proof}

\begin{Rq}
If we only assume that the connected component of $G$ is quasi-reductive, then we would think that $G(k) / S_{V}$ is still virtually abelian.
In fact, since $G^\circ$ is a Zariski-closed normal $k$-subgroup of $G$ of finite index, one get that $G^\circ(k)$ is a closed normal subgroup of $G(k)$.
But, because $S_{V}$ may not be a topologically characteristic subgroup of $G^\circ(k)$, the subgroup $S_{V}$ may not be normal in $G(k)$.
\end{Rq}

\section{Maximal pro-\texorpdfstring{$p$}{p} subgroups of a semisimple group}

In the group of rational points of a non-semisimple $k$-group, the fact that maximal bounded subgroups need not be compact may be an obstruction to the use of profinite group theory.
As an example of bad behaviour of non-semisimple groups, the maximal pro-$p$ subgroup of $\mathbb{G}_m(k)=k^\times$ is not finitely generated when $k = \mathbb{F}_q((t))$.
From now on, we reduce our study to the case of a semisimple $k$-group $G$ and we only consider smooth affine $k$-groups, that we will call algebraic $k$-group.

The conjugacy theorem \ref{thm:conjugaison:maximal:pro-p} is the generalisation to arbitrary characteristic of \cite[Theorem 3.10]{PlatonovRapinchuk}, which Platonov and Rapinchuk prove in characteristic $0$ and attribute to Matsumoto.
The proof is given in part \ref{subsection:conjugacy:thm}, using Bruhat-Tits buildings instead of maximal orders.

Furthermore, as we obtained a description of maximal profinite subgroups of $G(k)$ in Proposition \ref{prop:description:compact:maximal}, Theorem \ref{thm:description:maximal:pro-p} establishes an analogous description of maximal pro-$p$ subgroups.
It is proven in part \ref{subsection:description:building}.
In practice, the description by integral models established in Theorem \ref{thm:description:models:pro-p} is more convenient;
it is proven in part \ref{subsection:description:integral:models}.

\subsection{Proof of the conjugacy theorem}
\label{subsection:conjugacy:thm}

Let us first investigate the case of an algebraic group defined over a finite field.
This case corresponds to special fibers of integral $\mathcal{O}_k$-models (these models are useful in order to make a description of profinite subgroups).

\begin{Lem}
\label{lem:Sylow:over:finite:field}
Let $k$ be a finite field of characteristic $p$.
Let $H$ be a connected algebraic $k$-group.
The $p$-Sylow subgroups of the finite group $H(k)$ are exactly the groups $B_u(k)$
where $B$ is a Borel subgroup\footnote{By a theorem due to Lang \cite[16.6]{Borel}, we know that a linear algebraic group $H$ defined over a finite field $\kappa$ admits Borel subgroups themselves defined over $\kappa$.} of $H$ defined over $k$ and $B_u$ is the unipotent radical of $B$.

Moreover, the normalizer in $G$ of a $p$-Sylow subgroup is a Borel subgroup of $G$ and the map $B \mapsto B_u(k)$ is a bijection between the set of Borel $k$-subgroups of $H$ and the set of $p$-Sylow subgroups of $H(k)$.
\end{Lem}

\begin{proof}
Denote by $q$ the cardinal of $k$.
Let $P$ be a $p$-Sylow subgroup of $H(k)$.
Let $g \in P$ and $g = g_s \cdot g_u$ the Jordan decomposition of $g$.
Since $H$ is affine, there exists an integer $n \in \NN^*$ and a faithful linear representation $\rho : H \hookrightarrow GL_{n,k}$ \cite[5.1]{Borel} such that
$\rho(g_s) = \rho(g)_s$.
Hence, the order of this element divides $(q - 1)^n$, so it is prime to $p$.
As a consequence $g = g_u$.
Hence $P$ consists in unipotent elements of $H(k)$.
Since $k$ is perfect and $H$ is connected, by \cite[3.7]{BorelTits-unipotent}, there exists a Borel $k$-subgroup $B$ such that $P$ is contained in the group of rational points of the unipotent radical of $B$, denoted by $B_u(k)$.
Since $k$ is perfect, $B_u$ is $k$-split \cite[1.1.11]{BruhatTits2}.
Hence, $B_u(k)$ is a $p$-group.
Since $P$ is a $p$-Sylow subgroup of $H(k)$, we have $P = B_u(k)$ by maximality.

Since the Borel subgroups are $H(k)$-conjugate \cite[16.6]{Borel}, and since the $p$-Sylow subgroups of the finite group $H(k)$ are $H(k)$-conjugate, we obtain a surjective map $\Psi : B \mapsto B_u(k)$ between Borel $k$-subgroups of $H$ and $p$-Sylow subgroups of $H(k)$.
Let us show that it is a bijective map.

Fix $B$ a Borel $k$-subgroup of $H$ and $S$ a maximal $k$-split torus of $B$, hence of $H$.
Define $T = \mathcal{Z}_{H}(S)$, it is a maximal torus of $H$ defined over $k$ since an algebraic group over a finite field is quasi-split.
Since $k$ is perfect, the unipotent radical of $B$ is $k$-split.
The $k$-group $B$ has a Levi decomposition $B = T \cdot B_u$ \cite[C.2.4]{CGP}.

On the one hand, since $H(k)$ acts by conjugation on the set of Borel $k$-subgroups of $H$, the number of Borel $k$-subgroups is equal to the cardinal of $H(k) / \mathcal{N}_{H(k)}(B)$.
By a theorem due to Chevalley \cite[11.16]{Borel}, a Borel subgroup of $H$ is equal to its normalizer, hence $\mathcal{N}_{H(k)}(B) = B(k)$.
On the other hand, since $H(k)$ acts by conjugation on the set of its $p$-Sylow subgroups, the number of its $p$-Sylow subgroups is equal to the cardinal of $H(k) / \mathcal{N}_{H(k)}(B_u(k))$.

Hence, it suffices to show $\mathcal{N}_{H(k)}(B_u(k)) = B(k)$.
Denote by $N = \mathcal{N}_H(S)$ the normalizer of $S$ in $H$.
Since $N$ normalises $T$, we get that $N(k)$ normalises $T(k)$.
Since $B(k) = T(k) B_u(k) = B_u(k) T(k)$, by \cite[C.2.8]{CGP}, we get $G(k) = B_u(k) N(k) B_u(k)$.
Let $g \in \mathcal{N}_{H(k)}(B_u(k)) \subset H(k)$.
Write $g = unu'$ with ${u,u' \in B_u(k)}$ and $n \in N(k)$.
If $B = n B n^{-1}$, we have $n \in \mathcal{N}_H(B)(k) = B(k)$. Hence $g \in B(k)$.
By contradiction, suppose that $n \not\in T(k)$ and $B \neq n B n^{-1}$.
Thus the Weyl group ${_k}W = N(k) / T(k)$ is not trivial,
hence the group $H$ is not solvable and admits opposite root subgroups \cite[7.1.3, 7.1.5 and 7.2]{Springer}, which are $k$-split since $k$ is perfect \cite[15.5 (ii)]{Borel}.
Since $B$ and $nBn^{-1}$ are non equal Borel subgroup with the same maximal torus $T$, there exist a root $\alpha \in \Phi(B,T)=$ such that $n \cdot \alpha \not\in \Phi(B,T)$.
Let $v \in U_{\alpha}(k)$ be a non-trivial element. Since $U_{n \cdot \alpha} \cap B = \{1\}$, we have $n^{-1} v n \not\in B(k)$.
This contradicts $n = u^{-1} g {u'}^{-1} \in \mathcal{N}_{H(k)}(B_u(k))$.
Hence $\mathcal{N}_{H(k)}(B_u(k)) \subset B(k)$.

Moreover, since the $k$-group $\mathcal{N}_{H}(B_u(k))$ contains $B$, it is a parabolic subgroup of $H$. Since it does not contain opposite root subgroups, it is a minimal parabolic subgroup, hence a Borel subgroup of $H$.

As a consequence, the equality $\mathcal{N}_{H(k)}(B_u(k)) = \mathcal{N}_{H(k)}(B) = B(k)$ completes the proof.\qed
\end{proof}

\begin{Rq}
The bijective correspondence between Borel $k$-subgroups of $H$ and $p$-Sylow subgroups of $H(k)$ is useless in what follows.
We only need to know that the number of Borel $k$-subgroups is prime to $p$ (that is also a consequence of Bruhat decomposition).

Over a local field instead of a finite field, the fact that the normalizer of a $p$-Sylow subgroup of $H(k)$ is exactly $B(k)$ will be generalised by Proposition \ref{prop:normalizer:maximal:pro-p} with a simple connectedness assumption: normalizers of a maximal pro-$p$ subgroups are exactly Iwahori subgroups.
\end{Rq}

When a $p$-group acts on a finite set of cardinal prime to $p$, orbit-stabilizer theorem gives the existence of a fixed point. This statement can be generalised to the action of a pro-$p$ group.

\begin{Lem}\label{lem:action:finite:pro-p}
Let $p$ be a prime and $X$ a finite set of cardinal prime to $p$.
If $G$ is a pro-$p$ group acting continuously on $X$,
then $G$ fixes an element of $X$. 
\end{Lem}

\begin{proof}
For all $x \in X$, denote by $G_x$ the stabilizer of $x$.
Since $X$ is finite, $G_x$ is open.
Let $H = G_X = \bigcap_{x \in X} G_x$ be the subgroup of $G$ fixing $X$ pointwise.
Then $H$ is a normal open subgroup of $G$.
Hence $G/H$ is a $p$-group acting on $X$.
By the orbit-stabilizer theorem, $G/H$ fixes an element $x \in X$.
Hence $G$ fixes $x$.\qed
\end{proof}

Since a profinite subgroup is compact, by the Bruhat-Tits fixed point theorem, such a subgroup of $G(k)$ fixes a point $x_0 \in X(G,k)$.
Since the action of $G(k)$ preserves the structure of the simplicial complex, we get an action on the star of $x_0$, that means an action on the set of facets whose closure contains $x_0$.
Showing that the subset of alcoves of this set is a finite set of cardinal prime to $p$, we will get the following:

\begin{Prop}\label{prop:pro-p:stabilises:alcove}
A pro-$p$ subgroup of $G(k)$ setwise stabilises an alcove of $X(G,k)$.
\end{Prop}

\begin{proof}
Let $U$ be a pro-$p$ subgroup of $G(k)$.
By Proposition \ref{prop:description:compact:maximal}, there exists a point $y \in X(G,k)$ such that $\mathrm{Stab}_{G(k)}(y)$ is a maximal compact subgroup of $G(k)$ containing $U$.
Consider the (non-empty) set $\mathcal{C}_y$ of alcoves of $X(G,k)$ whose closure contains $y$.
Be careful that we forget the Euclidean structure provided by $X(G,k)$ and we only look at $\mathcal{C}_y$ as a discrete set.

Denote by $F$ the facet\footnote{Here, we consider the definition in which the facets form a partition of the building $X(G,k)$, so that any point is contained in a unique facet. Alternatively, we could consider that facets are closed polysimplices and then $F$ would be defined as the smallest facet of $X(G,k)$ containing $y$.} of $X(G,k)$ containing $y$.
By conjugation, assume that $F \subset \mathbb{A}$.
Define the star of $F$, denoted by $X(G,k)_F$, as the set of facets $F'$ of $X(G,k)$ such that $F \subset \overline{F'}$.
We endow this set with the partial order $F' \leq F'' \Leftrightarrow F' \subset \overline{F''}$.
Denote by $\mathfrak{G}_F$ the connected integral model of $G$ associated to $F$ (see definition in chapters \cite[4.6 and 5.1]{BruhatTits2}).
Denote by $\kappa$ the residue field of $k$ and consider $\mathcal{P}_F$ the set of $\kappa$-parabolic subgroups of $\overline{\mathfrak{G}_F}$ ordered by the inverse of the inclusion.
There is an isomorphism of ordered sets between $X(G)_F$ and $\mathcal{P}_F$ \cite[4.6.32 et 5.1.32 (i)]{BruhatTits2}
such that maximal simplices of $X(G)_F$ are exactly the elements of $\mathcal{C}_y$, and the minimal parabolic $\kappa$-subgroups of $\overline{\mathfrak{G}_F}$ correspond to them bijectively.
By Lang's theorem \cite[16.6]{Borel}, the minimal parabolic $\kappa$-subgroups of  $\overline{\mathfrak{G}_F}$ are exactly its Borel $\kappa$-subgroups.
By Lemma \ref{lem:Sylow:over:finite:field}, we obtain a bijection between $\mathcal{C}_y$ and the set of $p$-Sylow subgroups of $\mathfrak{G}_F(\kappa)$.

Since $G(k)$ preserves the poly-simplicial structure of $X(G,k)$ and $U$ fixes $y$, the group $U$ acts on $\mathcal{C}_y$.
For all $\mathbf{c},\mathbf{c'} \in \mathcal{C}_y$, by continuity of the action $G(k) \times X(G,k) \rightarrow X(G,k)$, the subset $\{g \in U\,,\,g \cdot \overline{\mathbf{c}} = \overline{\mathbf{c'}} \}$ is closed in $U$.
As a consequence, $U$ acts continuously on the finite set $\mathcal{C}_y$, whose cardinal is congruent to $1$ modulo $p$.
By Lemma \ref{lem:action:finite:pro-p}, $U$ fixes an alcove $\mathbf{c} \in \mathcal{C}_y$, hence $U$ setwise stabilises it in $X(G,k)$.\qed
\end{proof}

We provide an example of a pro-$p$ subgroup that stabilizes setwise an alcove but not pointwise.

\begin{Ex}\label{ex:adjoint:pro-p:group}
Let $p=3$ and $G = \mathrm{PGL}_{3,\mathbb{Q}_3}$.
The affine building of $G$ is of type $\widetilde{A_2}$.
Take the maximal torus $S=T=\left\{\left[\begin{matrix}x &0 &0\\0&y&0\\0&0&1\end{matrix}\right]\ ,\ x, y \in \mathbb{G}_{m,\mathbb{Q}_3}\right\}$ of $G$.
The element $t= \left[\begin{matrix}3 &0 &0\\0&1&0\\0&0&1\end{matrix}\right] \in T$ acts on $T$ by a translation \gts{of step $\frac{2}{3} (2 \alpha^\vee + \beta^\vee)$} where $\alpha, \beta$ are the simple roots :
indeed, it is the class of the element $\begin{pmatrix}3^{2/3}&0&0\\0&3^{-1/3}&0\\0&0&3^{-1/3}\end{pmatrix} \in \mathrm{SL}_{3,\mathbb{Q}_3(\sqrt[3]{3})}$ over the field $\mathbb Q_3(\sqrt[3]{3})$, that is a translation of step $1$ permuting cyclicly the types of vertices.
Let $x$ be a vertex. There exists an alcove $\mathbf{c}$ containing both $x$ and $t \cdot x$. There are two elements $n_1, n_2 \in N = \mathcal{N}_G(T)(\mathbb{Q}_3)$ both acting on the standard apartment as reflections so that $t \cdot \mathbf{c} = n_2 n_1 \cdot \mathbf{c}$. One get an element $g = n_1^{-1} n_2^{-1} t \in N(\mathbb{Q}_3)$ that acts on the standard apartment as the rotation of center the circumcenter $y$ of $\mathbf{c}$ and of order $3$. Thus, the subgroup of $\mathrm{PGL}_{3}(\mathbb{Q}_3)$ generated by $g$ and the maximal pro-$p$ subgroup of the stabilizer of $y$ is pro-$p$ and stabilizes exactly the alcove $\mathbf{c}$, but it does not fix it.

Note that a similar construction is also possible with $\mathrm{PGL}_2(\mathbb Q_2)$.
\end{Ex}

We now can give a proof of conjugation of maximal pro-$p$ subgroup theorem.

\begin{proof}[Proof of Theorem \ref{thm:conjugaison:maximal:pro-p}]
Let $U, U'$ be two maximal pro-$p$ subgroups of $G(k)$.
Let $\mathbf{c}, \mathbf{c}'$ be alcoves stabilized by the action of $U$ and $U'$ respectively (they exist by Proposition \ref{prop:pro-p:stabilises:alcove}).
Since $G(k)$ acts transitively on the set of alcoves of $X(G,k)$, there exists an element $g \in G(k)$ such that $g \cdot \mathbf{c}' = \mathbf{c}$.
Hence $g U' g^{-1}$ stabilises $\mathbf{c}$.
As a consequence, $U$ and $g U' g^{-1}$ are two maximal pro-$p$ subgroups of ${P = \mathrm{Stab}_{G(k)}(\mathbf{c})}$ which is compact by Lemma \ref{lem:parahoric:compact}(2).
Hence, $U$ and $gU'g^{-1}$ are two $p$-Sylow subgroups of the profinite group $P$. Since any two $p$-Sylow subgroups of a profinite group are conjugate \cite[1.4 Prop. 3]{SerreCohomologieGaloisienne}, $U$ and $g U' g^{-1}$ are conjugate in $P$,
so $U$ and $U'$ are conjugate in $G(k)$.\qed
\end{proof}

We now need to use root groups and integral models to prove the uniqueness of the alcove setwise stabilized by a given maximal pro-$p$ subgroup.
Theorem \ref{thm:description:maximal:pro-p} will be proven in part \ref{subsection:description:building}.

\subsection{Integral models} \label{subsection:description:integral:models}

In the proof of Proposition \ref{prop:pro-p:stabilises:alcove}, integral models were used; here, we will make a more systematic use of them.

Let $\Omega$ a non-empty bounded subset of the standard apartment $\mathbb{A}$.
Denote by $\pi_\kappa : \mathfrak{G}_\Omega^\dagger(\mathcal{O}_k) \rightarrow \mathfrak{G}_\Omega^\dagger(\kappa)$ the canonical reduction map.
Denote by $\overline{\mathfrak{G}}_\Omega = \left( \mathfrak{G}^\dagger_\Omega \right)_\kappa$ the special fiber.
Denote by $\left( \overline{\mathfrak{G}}_\Omega \right)^\circ$ the identity component of the $\kappa$-group $\overline{\mathfrak{G}}_\Omega$, and by $R_u(\overline{\mathfrak{G}}_\Omega^\circ)$ its unipotent radical,
defined over $\kappa$ because $\kappa$ is perfect \cite[0.7]{BorelTits}.
Denote by $\overline{\mathfrak{G}}_\Omega^{red} = \overline{\mathfrak{G}}_\Omega / R_u(\overline{\mathfrak{G}}_\Omega)$ the quotient $\kappa$-group (possibly non-connected since $\overline{\mathfrak{G}}_\Omega$ may be not connected).
The root system of its identity component is the set $\Phi_\Omega$ of roots $a \in \Phi$, where $\Phi$ denotes the relative root system of $G$, such that the root $a$ seen as an affine map is constant over $\Omega$ and has values in the set $\Gamma'_a$ \cite[10.36]{Landvogt}.
Note that, when $\Omega$ contains an alcove, no root of $\Phi$ is constant on $\Omega$ since an alcove of $\mathbb{A}$ is open in $\mathbb{A}$, hence $\Phi_\Omega$ is empty.

Denote by $\pi_q : \overline{\mathfrak{G}}_\Omega \rightarrow \overline{\mathfrak{G}}_\Omega^{red}$ the quotient $\kappa$-morphism of algebraic $\kappa$-groups,
and, by notation abuse, $\pi_q : \mathfrak{G}_\Omega^\dagger(\kappa) \rightarrow \overline{\mathfrak{G}}_\Omega^{red}(\kappa)$ the homomorphism of abstract groups deduced from $\pi_q$.
It will be clear from the context which of these two morphisms will be considered.

\begin{Not}
Identifying the abstract groups $\mathfrak{G}_\Omega^\dagger(\kappa) = \overline{\mathfrak{G}}_\Omega(\kappa)$, we can define the composite morphism $\pi_\Omega = \pi_q \circ \pi_\kappa$.
Denote by $P^+_\Omega$ the kernel of $\pi_\Omega$.

More specifically, if $F$ is a facet of the building $X(G,k)$, by transitivity, there exists an element $g \in G(k)$ such that $g \cdot F \subset \mathbb{A}$.
Denote $P^+_F = g^{-1} P^+_{g \cdot F} g$.
This group does not depend on the choice of such a $g$.
\end{Not}

The goal is to show that, when $G$ is simply connected, $P^+_F$ is a maximal pro-$p$ subgroup of the profinite (by Lemma \ref{lem:parahoric:compact}(2)) subgroup $\mathrm{Stab}_{G(k)}(F)$.
Note that with this notation, it is not required that the facet $F$ be contained in the standard apartment $\mathbb{A}$.

\begin{Lem}\label{lem:canonical:morphism}
The morphism $\pi_\kappa$ is a surjective group homomorphism and its kernel $\ker \pi_\kappa$ is a pro-$p$ group.
\end{Lem}

\begin{proof}
Surjectivity of $\pi_\kappa$ is a consequence of smoothness of the $\mathcal{O}_k$-model $\mathfrak{G}_\Omega^\dagger$ \cite[2.3 Prop. 5]{NeronModels}.

The smooth affine $\mathcal{O}_k$-group of finite type $\mathfrak{G}_\Omega^\dagger$ has a faithful linear representation, that means a closed immersion, $\rho : \mathfrak{G}_\Omega^\dagger \rightarrow \mathcal{GL}_{n,\mathcal{O}_k}$ for which it corresponds a surjective Hopf $\mathcal{O}_k$-algebras homomorphism $\varphi : A \twoheadrightarrow B$ where $A$ and $B$ denote respectively the $\mathcal{O}_k$-Hopf algebras of $\mathcal{GL}_{n,\mathcal{O}_k}$ and $\mathfrak{G}_\Omega^\dagger$.
Denote by $\widetilde{\pi_\kappa} : \mathcal{GL}_{n,\mathcal{O}_k}(\mathcal{O}_k) \rightarrow \mathcal{GL}_{n,\mathcal{O}_k}(\kappa)$ the canonical surjective homomorphism (defined as $\pi_\kappa$ above).
Hence $\ker \pi_\kappa = \{ f : B \rightarrow \mathcal{O}_k , f \otimes 1 = \varepsilon \otimes 1 \}$ and
$\ker \widetilde{\pi_\kappa} = \{ f : A \rightarrow \mathcal{O}_k , f \otimes 1 = \widetilde{\varepsilon} \otimes 1 \}$
where $\varepsilon$ (resp. $\widetilde{\varepsilon}$) is the co-unit of $B$ (resp. A).

On $\mathcal{O}_k$ points, we have $\ker \widetilde{\pi_\kappa} = \mathrm{GL}_n(\mathfrak{m})$, according to the notation of the proof of Lemma \ref{lem:existence:open:pro-p}.
Since $\widetilde{\varepsilon} = \varphi^* \varepsilon$,
we have the following commutative diagram:
\begin{equation*}
  \xymatrix{
    0 \ar[r] &
    \ker \pi_\kappa \ar[r]^{\subset} \ar@{^{(}-->}[d] &
    \mathfrak{G}_\Omega^\dagger(\mathcal{O}_k) \ar@{->>}[r]^{\pi_\kappa} \ar@{^{(}->}[d]^{\rho_{\mathcal{O}_k}} &
    \mathfrak{G}_\Omega^\dagger(\kappa) \ar[r] \ar@{^{(}->}[d]^{\rho_{\kappa}} &
    1 \\
    0 \ar[r] &
    \ker \widetilde{\pi_\kappa} \ar[r]^{\subset} &
    \mathrm{GL}_n(\mathcal{O}_k) \ar@{->>}[r]^{\widetilde{\pi_\kappa}} &
    \mathrm{GL}_n(\kappa) \ar[r] &
    1
  }
\end{equation*}

Hence $\ker \pi_\kappa$ is isomorphic to a closed subgroup of $\ker \widetilde{\pi_\kappa}$, so it is a pro-$p$ group.\qed
\end{proof}

\begin{Prop}
\label{prop:kernel:pro-p}
The group $P^+_\Omega$ is a normal pro-$p$ subgroup of $\mathfrak{G}_\Omega^\dagger(\mathcal{O}_k)$.
\end{Prop}

\begin{proof}
By Lemma \ref{lem:exact:sequence:cohomology}(b), we have $\ker \pi_q = R_u(\overline{\mathfrak{G}}_\Omega)(\kappa)$,
hence it is a $p$-group as a group of rational points of a unipotent $\kappa$-group.
The following sequence of group homomorphism $1 \longrightarrow \ker \pi_\kappa \stackrel{\subseteq}{\longrightarrow} \ker ( \pi_q \circ \pi_\kappa ) \stackrel{\pi_\kappa}{\longrightarrow} \ker \pi_q \stackrel{\pi_q}{\longrightarrow} 1$ is exact.
Indeed, check that $\pi_\kappa(\ker \pi_q \circ \pi_\kappa) = \ker \pi_q$.

If $g \in \pi_\kappa(\ker \pi_q \circ \pi_\kappa)$, then there exists $h \in \ker \pi_q \circ \pi_\kappa$ such that $g = \pi_\kappa(h)$.
hence $\pi_q(g) = \pi_q \circ \pi_\kappa(h) = 1$, and so $g \in \ker \pi_q$.

Conversely, if $g \in \ker \pi_q$, by surjectivity of $\pi_\kappa$ (given by Lemma \ref{lem:canonical:morphism}), there exists $h \in \mathfrak{G}_\Omega(\mathcal{O}_k)$ such that $\pi_\kappa(h) = g$.
Hence $\pi_q \circ \pi_\kappa(h) = \pi_q(g) = 1$, and so $h \in \ker (\pi_q \circ \pi_\kappa)$.
Hence $g \in \pi_\kappa(\ker (\pi_q \circ \pi_\kappa))$.

As a consequence, $P^+_\Omega = \ker \pi_\Omega$ is a pro-$p$ group.\qed
\end{proof}

\begin{Lem}\label{lem:normal:p-group:finite:field}
Let $k$ be a finite field of characteristic $p$.
If $H$ is a reductive $k$-group, then $H(k)$ does not have a non-trivial normal $p$-subgroup.
\end{Lem}

\begin{proof}
Let $P$ be a normal $p$-subgroup of $H(k)$.
It is a subgroup of a $p$-Sylow subgroup of $H(k)$.
By Lemma \ref{lem:Sylow:over:finite:field}, there exists a Borel $k$-subgroup $B$ such that $P \subset \mathcal{R}_u(B)(k)$.

Let $S$ be a maximal $k$-split torus of $H$.
Denote $T = \mathcal{Z}_H(S)$, it is a maximal torus of $H$ defined over $k$ and contained in $B$.
Let $n \in \mathcal{N}_H(T)(k)$ such that $B$ and $n B n^{-1}$ are opposite Borel $k$-subgroups.
Hence, $B \cap n B n^{-1} = T$ \cite[14.1]{Borel} is a torus.
We have $n P n^{-1} = P$ because $P$ is normal in $H(k)$.
Hence, $P$ is a subgroup of $T(k)$ and $\#T(k)$ is prime to $p$.
As a consequence $P \subset T(k)$ is trivial.\qed
\end{proof}

To obtain results about the maximality of $\ker \pi_\Omega$, we require that $\pi_\Omega$ is surjective.

\begin{Lem}\label{lem:pi:surjective}
The morphism of abstract groups $\pi_\Omega$ is surjective.

In particular, if $Q$ is a $p$-Sylow subgroup of  $\overline{\mathfrak{G}}_\Omega^{\mathrm{red}}(\kappa)$, then $\pi_\Omega^{-1}(Q)$ is a maximal pro-$p$ subgroup of $\mathfrak{G}_\Omega^\dagger(\mathcal{O}_k)$.
\end{Lem}

\begin{proof}
A finite field is perfect, hence by \cite[III.2.1 Prop. 6]{SerreCohomologieGaloisienne} applied to the connected ($\kappa$-split) unipotent $\kappa$-group  $U = R_u(\overline{\mathfrak{G}_\kappa})$, we have $H^1(\kappa,U) = 0$.
Hence by \cite[I.5.5 Prop.38]{SerreCohomologieGaloisienne} the morphism of abstract groups $\pi_q$ is surjective.
According to Lemma \ref{lem:canonical:morphism}, the composite morphism $\pi_\Omega$ is surjective.

By Proposition \ref{prop:kernel:pro-p}, the surjective morphism $\pi_\Omega$ has a pro-$p$ kernel. Hence, for every $p$-subgroup $Q$ of $\overline{\mathfrak{G}}_\Omega^{\mathrm{red}}(\kappa)$, the group $\pi_\Omega^{-1}(Q)$ is pro-$p$ (as an extension of such groups).
Hence, if $Q$ is a $p$-Sylow subgroup, then $\pi_\Omega^{-1}(Q)$ is a maximal pro-$p$ subgroup.\qed
\end{proof}

\begin{Prop}\label{prop:maximal:normal:pro-p}
If $\overline{\mathfrak{G}_\Omega}$ is connected,
then the kernel $P^+_\Omega$ is a maximal normal pro-$p$ subgroup of $\mathfrak{G}_\Omega^\dagger(\mathcal{O}_k)$.
\end{Prop}

\begin{proof}
Let $\widetilde{P}$ be a normal pro-$p$ subgroup of $\mathfrak{G}_\Omega(\mathcal{O}_k)$ containing $P^+_\Omega$.
By \cite[I.1.4 Prop.4]{SerreCohomologieGaloisienne}, its image by the surjective morphism $\pi_\Omega$ (see Lemma \ref{lem:pi:surjective}) is a normal $p$-subgroup of $\overline{\mathfrak{G}}_\Omega^{\mathrm{red}}(\kappa)$.

When $\overline{\mathfrak{G}}_\Omega$ is connected, the quotient $\overline{\mathfrak{G}}_\Omega^{\mathrm{red}}$ is a connected reductive $\kappa$-group.
Hence, by Lemma \ref{lem:normal:p-group:finite:field}, $\pi(\widetilde{P})$ is trivial and $\widetilde{P} = P^+_\Omega$.\qed
\end{proof}

\subsubsection*{Under simple connectedness assumption}

From now on, assume that the semisimple $k$-group $G$ is simply connected.
Hence $\mathfrak{G}_\Omega^\dagger = \mathfrak{G}_\Omega$ \cite[4.6.32 and 5.1.31]{BruhatTits2}.

\begin{Prop}
\label{prop:kernel:maximal:pro-p:simply:connected:case}
Assume $\Omega = \mathbf{c} \subset \mathbb{A}$ is an alcove and $G$ is simply connected.
Then $P^+_\Omega$ is a maximal pro-$p$ subgroup of $\mathfrak{G}_\Omega(\mathcal{O}_k)$.
\end{Prop}

First, recall the following fact, given by Tits \cite[3.5.2]{TitsCorvallis}:

\begin{Lem}\label{lem:connectedness:of:special:fiber}
Under above assumptions and notations,
the algebraic $\kappa$-group $\overline{\mathfrak{G}}_\Omega$ is connected.
\end{Lem}

\begin{proof}[Proof of Proposition \ref{prop:kernel:maximal:pro-p:simply:connected:case}]
Since $\Omega$ is an alcove, the root system of $\overline{\mathfrak{G}}_\Omega / R_u(\overline{\mathfrak{G}}_\Omega)$ is empty \cite[4.6.12(i), 5.1.31]{BruhatTits2}.
By Lemma \ref{lem:connectedness:of:special:fiber}, $\overline{\mathfrak{G}}_\Omega^{\mathrm{red}}$ is a connected reductive quasi-split $\kappa$-group with a trivial root system.
Hence, it is a $\kappa$-torus and so, does not have a non-trivial $p$-subgroup.
Hence, for every pro-$p$ subgroup $P$ of $\mathfrak{G}_\Omega^\dagger(\mathcal{O}_k) = \mathfrak{G}_\Omega(\mathcal{O}_k)$, the image $\pi_\Omega(P)$ by the surjective morphism $\pi_\Omega$ (Lemma \ref{lem:pi:surjective}) is a $p$-group \cite[1.4 Prop.4]{SerreCohomologieGaloisienne}, hence trivial.
As a consequence, the kernel $P^+_\Omega$ is the (unique) maximal pro-$p$ subgroup of $\mathfrak{G}_\Omega(\mathcal{O}_k)$.\qed
\end{proof}

Now, one can give a proof of Theorem \ref{thm:description:models:pro-p}.

\begin{proof}[Proof of Theorem \ref{thm:description:models:pro-p}]
Let $P$ a maximal pro-$p$ subgroup.
By Proposition \ref{prop:pro-p:stabilises:alcove}, we have $P \subset \mathrm{Stab}_{G(k)}(\mathbf{c})$.
Let $\mathbf{c}_0 \subset \mathbb{A}$.
By strong transitivity of $G(k)$ on the building $X(G,k)$, there exists $g \in G(k)$ such that $g \mathbf{c}_0 = \mathbf{c}$.
Hence, $g^{-1} P g$ is a maximal pro-$p$ subgroup of $\mathfrak{G}_{\mathbf{c}_0}(\mathcal{O}_k)$.
By Proposition \ref{prop:kernel:maximal:pro-p:simply:connected:case}, we have $P = g P^+_{\mathbf{c}_0} g^{-1}$.\qed
\end{proof}

\subsubsection*{Valued root group datum in the quasi-split simply connected case}

To conclude in the simply connected case, let us interpret this group in terms of a valued root group datum.
This could be a bit tricky in the general case and, in the two next propositions, we assume that $G$ is, moreover, a quasi-split semisimple $k$-group.
In a further work \cite{Loisel-GenerationFrattini}, we compute the Frattini subgroup of a maximal pro-$p$-subgroup by the explicit decomposition of Proposition \ref{prop:explicit:decomposition}.

\begin{Prop}\label{prop:explicit:decomposition}
Let $G$ be a quasi-split simply connected semisimple group defined over a local field $k$ of residual characteristic $p$.
Let $S$ be a maximal $k$-split torus, $T = \mathcal{Z}_G(S)$ be the associated maximal $k$-torus and $\mathbf{c}$ be an alcove of the apartment associated to $T$.
Let $P^+_\mathbf{c}$ be the maximal pro-$p$ subgroup of $G(k)$ that fixes $\mathbf{c}$.
The group $P^+_\mathbf{c}$ admits the following directly generated product structure:
\begin{equation*}\label{eqn:directly:generated:product}
P_{\mathbf{c}}^+ = \left( \prod_{a \in \Phi_{\mathrm{nd}}^+} U_{-a,f_\mathbf{c}(-a)}  \right) \cdot T(k)_b^+ \cdot  \left( \prod_{a \in \Phi_{\mathrm{nd}}^+} U_{a,f_\mathbf{c}(a)} \right)
\end{equation*}
where $T(k)_b^+$ is the (unique) maximal pro-$p$ subgroup of $T(k)$ and $\Phi_{\mathrm{nd}}$ denotes the non-divisible roots of the relative $k$-root system $\Phi(G,S)$.

In particular, $T(k)_b^+ = P^+_{\mathbf{c}} \cap T(k)_b$ where $T(K)_b$ is the maximal compact subgroup of $T(k)$.
\end{Prop}

\begin{proof}
By the simple connectedness assumption, Proposition 3.5 of \cite{Landvogt} gives $\mathfrak{T}(\mathcal{O}_k) = T(k)_b$ where $\mathfrak{T}$ denotes the integral model of $T$ defined in \cite[§3]{Landvogt}.
By definition of the integral models \cite[4.3.2 and 4.3.5]{BruhatTits2}, for any root $a \in \Phi$, we have $\mathfrak{U}_{a,\mathbf{c}}(\mathcal{O}_k) = U_{a,\mathbf{c}}$  and, if $2a \in \Phi$, we have $U_{2a,\mathbf{c}} \subset U_{a,\mathbf{c}}$. Therefore, we have $\mathfrak{U}_{-\Phi^+,\mathbf{c}}(\mathcal{O}_k) = \left( \prod_{a \in \Phi_{\mathrm{nd}}^+} U_{-a,f_\mathbf{c}(-a)}  \right)$ and $\mathfrak{U}_{\Phi^+,\mathbf{c}}(\mathcal{O}_k) = \left( \prod_{a \in \Phi_{\mathrm{nd}}^+} U_{a,f_\mathbf{c}(a)} \right)$ by \cite[4.6.3]{BruhatTits2}.

By \cite[4.6.7]{BruhatTits2}, we have $\mathfrak{G}_\mathbf{c}(\mathcal{O}_k) = \mathfrak{U}_{\Phi^+,\mathbf{c}}(\mathcal{O}_k) \cdot \mathfrak{U}_{\Phi^-,\mathbf{c}}(\mathcal{O}_k) \cdot N_\mathbf{c}$
where $N_\mathbf{c} = \left\{g \in \mathcal{N}_G(S)(k)\ ,\ \forall x \in \mathbf{c}\ g \cdot x = x \right\}$.
By \cite[6.4.9 (iv)]{BruhatTits1}, $N_\mathbf{c}$ is the group generated by the $N_{a,\mathbf{c}} = N_\mathbf{c} \cap \langle U_{-a,\mathbf{c}} , U_{a,\mathbf{c}} \rangle$ for $a \in \Phi$.
Since $\mathbf{c}$ is an alcove, for any relative root $a \in \Phi$, we have $f_\mathbf{c}(a) + f_\mathbf{c}(-a) > 0$.
Hence, by \cite[8.6 (i)]{Landvogt}, we have $N_{a,\mathbf{c}} \subset T(k)_b$ .
Since $N_\mathbf{c}$ contains $T(k)_b$, we have the equality $N_\mathbf{c} = T(k)_b$.
Thus, we have $\mathfrak{G}_\mathbf{c}(\mathcal{O}_k) = \mathfrak{U}_{\Phi^-,\mathbf{c}}(\mathcal{O}_k) \cdot \mathfrak{T}(\mathcal{O}_k) \cdot \mathfrak{U}_{\Phi^+,\mathbf{c}}(\mathcal{O}_k)$.

In the proof of Proposition \ref{prop:kernel:maximal:pro-p:simply:connected:case}, we have seen that $\overline{\mathfrak{G}}_\mathbf{c}^{\mathrm{red}}(\kappa)$ does not have a non-trivial $p$-subgroup.
Hence $\mathfrak{U}_{\pm\Phi^+,\mathbf{c}}(\mathcal{O}_k) \subset \ker \pi_\mathbf{c} = P^+_\mathbf{c}$ since the image of a pro-$p$ group by a surjective continuous morphism is a pro-$p$ group.
Thus, we obtain the equality.\qed
\end{proof}

By quasi-splitness and simple connectedness, the maximal $k$-torus $T$ is an induced torus \cite[4.4.16]{BruhatTits2}, generated by coroots, and we can be more precise about the above description by root group datum:

\begin{Prop}
Let $G$ be a quasi-split simply connected semisimple group defined over a local field $k$ of residual characteristic $p$.
Let $S$ be a maximal $k$-split torus and $T = \mathcal{Z}_G(S)$ be the associated maximal $k$-torus.
Let $\Delta$ be a basis of the relative root system $\Phi = \Phi(G,S)$.
There is the following isomorphism of topological groups:
\begin{equation}
\begin{array}{cccc}
\prod_{a \in \Delta} \hat{a}^\vee : & \prod_{a \in\Delta} (1+ \mathfrak{m}_{l_a}) & \rightarrow & T(k)_b^+\\
& (t_a)_{a \in \Delta} & \mapsto & \prod_{a \in \Delta} \hat{a}^\vee(t_a)\end{array}
\end{equation}
where $\hat{a} = 2a$ if $2a \in \Phi$, and $\hat{a} = a$ otherwise;
$L_a$ denotes the minimal field of definition of the root $a$ (defined in \cite[4.1.3]{BruhatTits2})
and $\mathfrak{m}_{L_a}$ denotes the maximal ideal of its ring of integers.
\end{Prop}

\begin{proof}
Since $G$ is a simply connected quasi-split semisimple $k$-group, by \cite[4.4.16]{BruhatTits2}, $T$ is an induced torus and, more precisely, there is the following isomorphism $\prod_{a \in \Delta} \hat{a}^\vee :  \prod_{a \in \Delta} R_{L_a/K}(\mathbb{G}_{m,L_a}) \simeq T$, where $\Delta$ denotes a basis of the relative root system $\Phi$.
By uniqueness, up to isomorphism, of the $\mathcal{O}_k$-model, $\mathfrak{T}$ is $\mathcal{O}_k$-isomorphic to $\prod_{a \in \Delta} R_{\mathcal{O}_{L_a}/\mathcal{O}_k}(\mathbb{G}_{m,\mathcal{O}_{L_a}})$.
Hence, there is a natural isomorphism $\prod_{a \in \Delta}\mathcal{O}_{L_a}^\times \simeq \mathfrak{T}(\mathcal{O}_k) = T(k)_b$ of topological abelian groups,
and the maximal pro-$p$ subgroup is isomorphic to the direct product $\prod_{a \in \Delta} (1 + \mathfrak{m}_{L_a})$.\qed
\end{proof}

\subsection{Description using the action on a building} \label{subsection:description:building}

We now can derive the useful description of a maximal pro-$p$ subgroup of $G(k)$, as a pro-$p$-Sylow of the setwise stabilizer of a suitable alcove.
To prove Theorem \ref{thm:description:maximal:pro-p}, it suffices to show that every maximal pro-$p$ subgroup of $G(k)$ can be realised as such a group.

\begin{proof}[Proof of Theorem \ref{thm:description:maximal:pro-p}]
Let $P$ be a maximal pro-$p$ subgroup of $G(k)$.
By Proposition \ref{prop:pro-p:stabilises:alcove},
there exists an alcove $\mathbf{c}$ such that $P$ setwise stabilizes  $\mathbf{c}$.
By strong transitivity, we can and do assume that $\mathbf{c} \subset \mathbb{A}$.
In particular, $P$ is a maximal pro-$p$ subgroup of $\mathfrak{G}_{\mathbf{c}}^\dagger(\mathcal{O}_k)$.

Firstly, we show the uniqueness of such an alcove $\mathbf{c}$.
By Lemma \ref{lem:parahoric:compact}, the topological group $\mathfrak{G}_{\mathbf{c}}^\dagger(\mathcal{O}_k)$ is compact, hence profinite.
By Sylow theorem for profinite groups \cite[1.4 Prop.3 et 4 (a)]{SerreCohomologieGaloisienne}, there exists $g_0 \in \mathfrak{G}_{\mathbf{c}}^\dagger(\mathcal{O}_k)$ such that $P$ contains $g_0 P^+_\mathbf{c} g_0^{-1} = P^+_\mathbf{c}$.
It suffices to show that $P^+_\mathbf{c}$ does not stabilises any alcove of $X(G,k)$ different from $\mathbf{c}$.

For all $a \in \Phi$, the image  by $\pi_\mathbf{c}$ of the root group $U_{a,\mathbf{c}}(\mathcal{O}_k)$ is trivial because $\overline{\mathfrak{U}}_{a,\mathbf{c}}$ is a root group of $\overline{\mathfrak{G}}^{red}_\mathbf{c}$ \cite[10.34]{Landvogt}, hence trivial because $\mathbf{c}$ is an alcove \cite[10.36]{Landvogt}.
Hence $P^+_\mathbf{c}$ contains the subgroup $U_\mathbf{c}$ of $G(k)$ generated by $U_{a,\mathbf{c}}$ for every $a \in \Phi$.
The group $P^+_\mathbf{c}$ acts on the set of all facets of $X(G,k)$ not contained in $\mathop{cl}(\mathbf{c})$ since it setwise stabilizes $\mathop{cl}(\mathbf{c})$ and preserves the simplicial structure of $X(G,k)$.
Let $F$ be such a facet.
Let $A'$ be an apartment containing $\mathbf{c}$ and $F$.
Let $A''$ be an apartment containing $\mathbf{c}$ but not $F$.
Since the group $U_\mathbf{c}$ acts transitively on the set of apartments containing $\mathbf{c}$ \cite[13.7]{Landvogt},
there exists $u \in U_\mathbf{c} \subset P^+_\mathbf{c}$ such that $u \cdot A' = A''$.
Hence $P^+_\mathbf{c}$ does not stabilize $F$.

Conversely, let $\mathbf{c}$ be an alcove of $X(G)$ and $P$ be a maximal pro-$p$ subgroup of $\mathrm{Stab}_{G(k)}(\mathbf{c})$.
Let $P'$ be a maximal pro-$p$ subgroup of $G(k)$ containing $P$.
Such a $P'$ exists by Lemma \ref{lem:link:with:openness} and Proposition \ref{prop:compact:open:finitely:contained}.
Let $\mathbf{c}'$ be the unique alcove stabilized by $P'$, hence by $P$.
Since $P$ contains $P^+_{\mathbf{c}}$ according to Lemma \ref{lem:pi:surjective},
it does not stabilize any facet of $X(G,k)$ out from $\mathop{cl}(\mathbf{c})$.
Hence $\mathbf{c} = \mathbf{c}'$ and $P'$ is a maximal pro-$p$ subgroup of $\mathrm{Stab}_{G(k)}(\mathbf{c})$.
By maximality of $P$, we have $P' = P$.\qed
\end{proof}

\begin{Cor}
If $G$ is a simply connected semisimple $k$-group, then $P$ is a maximal pro-$p$ subgroup of $G(k)$ if, and only if, there exists an alcove $\mathbf{c}$ of $X(G,k)$ such that $P = P^+_\mathbf{c}$.
Moreover, such an alcove $\mathbf{c}$ is uniquely determined by $P$ and the set of fixed points by $P$ in $X(G,k)$ is exactly the simplicial closure $\mathop{cl}(\mathbf{c})$.
\end{Cor}

\begin{proof}
The first part is a consequence of Proposition \ref{prop:kernel:maximal:pro-p:simply:connected:case} and of the first part of Theorem \ref{thm:description:models:pro-p}.

When $G$ is simply connected, the stabilizer of an alcove is also its pointwise stabilizer \cite[5.2.9]{BruhatTits2}.
This and Theorem \ref{thm:description:models:pro-p} gives the second part.\qed
\end{proof}

\begin{Rq}
If $p$ does not divide the order of the group of automorphisms of an alcove, then a maximal pro-$p$ subgroup has to fix the alcove $\mathbf{c}$ that it stabilizes. Therefore, it can be written $P = P^+_\mathbf{c}$.

The example \ref{ex:adjoint:pro-p:group} provides a counter-example for an adjoint semisimple group.
\end{Rq}

\subsubsection*{Iwahori subgroups in the simply connected case}

Recall the following definitions \cite[5.2]{BruhatTits2}

\begin{Def}~

(1) Given a facet $F$ of $X(G,k)$, call \textbf{connected pointwise stabilizer} of $F$ the subgroup $\mathfrak{G}_F(\mathcal{O}_k)$ of $G(k)$.

(2) A subgroup of $G(k)$ is called a \textbf{parahoric} (resp. \textbf{Iwahori}) subgroup if, and only if, it is the connected pointwise stabilizer of a facet (reps. an alcove) of $X(G,k)$.
\end{Def}

To conclude this study of pro-$p$ subgroups, the following proposition, given by \cite[§3.4]{PlatonovRapinchuk}, is a kind of generalisation of Lemma \ref{lem:Sylow:over:finite:field}.

\begin{Prop}\label{prop:normalizer:maximal:pro-p}
Assume that $G$ is simply connected.
A subgroup of $G(k)$ is an Iwahori subgroup if, and only if, it is the normalizer in $G(k)$ of a maximal pro-$p$ subgroup of $G(k)$.
\end{Prop}

\begin{proof}
Let $\mathbf{c}$ be an alcove of $\mathbb{A}$, let $g \in G(k)$ an element and $H$ the stabilizer of $g \cdot \mathbf{c}$.
Since the semisimple $k$-group $G$ is simply connected,
the stabilizer $H$ is in fact an Iwahori subgroup \cite[5.2.9]{BruhatTits2}.
By Proposition \ref{prop:kernel:pro-p}, $g P^+_{\mathbf{c}} g^{-1}$ is a normal pro-$p$ subgroup of $H$.
Hence $H \subset \mathcal{N}_{G(k)}(g P^+_\mathbf{c} g^{-1})$.
For every element $h \in \mathcal{N}_{G(k)}(g P^+_\mathbf{c} g^{-1})$, every $u \in P^+_\mathbf{c}$ and $x \in \mathbf{c}$,
one has $h^{-1}uh \cdot x = x$ because $g P^+_\mathbf{c} g^{-1}$ fixes $g \cdot \mathbf{c}$ pointwise.
Hence $h \cdot x $ is a point in $X(G,k)$ fixed by $P^+_\mathbf{c}$,
so $h \cdot x \in \mathbf{c}$ since it cannot be contained on the boundary of $\mathbf{c}$.
Since the action of $G(k)$ preserves the simplicial structure of $X(G,k)$,
the element $h$ stabilises $\mathbf{c}$.
Hence $\mathcal{N}_{G(k)}(g P^+_\mathbf{c} g^{-1}) = H$.
By Theorem \ref{thm:description:models:pro-p}, it gives the first implication.

Conversely,
let $U$ be a maximal pro-$p$ subgroup of $G(k)$.
Define $H = \mathcal{N}_{G(k)}(U)$.
Denote by $\mathbf{c}$ be the unique alcove fixed by $U$ given by Theorem \ref{thm:description:maximal:pro-p}.
By uniqueness of $\mathbf{c}$, the subgroup $H$ stabilises $\mathbf{c}$.
By Proposition \ref{prop:kernel:maximal:pro-p:simply:connected:case} (and conjugation), $U$ is a normal subgroup of $\mathrm{Stab}_{G(k)}(\mathbf{c})$.
Hence $H = \mathcal{N}_{G(k)}(U) = \mathrm{Stab}_{G(k)}(\mathbf{c})$ is an Iwahori subgroup of $G(k)$.\qed
\end{proof}

\begin{Cor}
Iwahori subgroups of $G(k)$ are $G(k)$-conjugate.
\end{Cor}

\begin{proof}
This is \cite[3.7]{TitsCorvallis}.
It is immediate by Theorem \ref{thm:conjugaison:maximal:pro-p} and Proposition \ref{prop:normalizer:maximal:pro-p}.\qed
\end{proof}

An interest of Proposition \ref{prop:normalizer:maximal:pro-p} is to have an \gts{intrinsic} definition (from the group theory point of view, in other words a description not using the action on the Bruhat-Tits building) of Iwahori subgroups in good cases (e.g. a simply connected group over a local field).
This provides a quick way to describe the affine Tits system in purely group-theoretic terms.





\begin{thebibliography}{00}
\bibitem[AB08]{AbramenkoBrown}
P.~Abramenko and K.~S. Brown,
\newblock {\em Buildings Theory and Applications},
\newblock Vol. 248 of {\em Graduate Texts in Mathematics},
\newblock Springer, New York, 2008.

\bibitem[BG17]{BaderGelander}
U.~Bader and T.~Gelander,
\newblock {\em Equicontinuous actions of semisimple groups},
\newblock Groups, Geometry, and Dynamics {\bf 11} (2017) no. 3, 1003--1039.

\bibitem[BH99]{BridsonHaefliger}
M.~R. Bridson and A.~Haefliger,
\newblock {\em Metric Spaces of non-Positive Curvature},
\newblock Vol. 319 of {\em Grundlehren der Mathematischen Wissenschaften},
\newblock Springer-Verlag, Berlin, 1999.

\bibitem[BLR90]{NeronModels}
S.~Bosch, W.~L{\"u}tkebohmert, and M.~Raynaud,
\newblock {\em N\'eron Models},
\newblock Vol.~21 of {\em Ergebnisse der Mathematik und ihrer Grenzgebiete (3)},
\newblock Springer-Verlag, Berlin, 1990.

\bibitem[Bo91]{Borel}
A.~Borel,
\newblock {\em Linear Algebraic Groups},
\newblock Vol.~126 of {\em Graduate Texts in Mathematics},
\newblock Springer-Verlag, New York, second edition, 1991.

\bibitem[BoT65]{BorelTits}
A.~Borel and J.~Tits,
\newblock {\em Groupes r\'eductifs},
\newblock Publications Math\'ematiques de l'IH\'ES {\bf 27} (1965), 55--150.

\bibitem[BoT71]{BorelTits-unipotent}
A.~Borel and J.~Tits,
\newblock {\em \'{E}l\'ements unipotents et sous-groupes paraboliques de groupes r\'eductifs. {I}},
\newblock Inventiones Mathematicae {\bf 12} (1971), 95--104.

\bibitem[BoT73]{BorelTits-HomoAbstraits}
A.~Borel and J.~Tits,
\newblock {\em Homomorphismes ``abstraits'' de groupes alg\'ebriques simples},
\newblock Annals of Mathematics {\bf 97} (1973), no. 3, 499--571.

\bibitem[BoT78]{BorelTitsNoteauCRAS}
A.~Borel and J.~Tits,
\newblock {\em Th\'eor\`emes de structure et de conjugaison pour les groupes alg\'ebriques lin\'eaires},
\newblock Comptes rendus de l'Acad\'emie des sciences, S\'erie A {\bf 287} (1978), no. 1, 55--57.

\bibitem[Bou81]{Bourbaki4-6}
N.~Bourbaki,
\newblock {\em \'El\'ements de Math\'ematique, Groupes et Alg\`ebres de Lie. Chapitres 4, 5 et 6},
\newblock Masson, Paris, 1981.

\bibitem[Bro89]{Brown}
K.~S. Brown,
\newblock {\em Buildings},
\newblock Springer-Verlag, New York, 1989.

\bibitem[BrT72]{BruhatTits1}
F.~Bruhat and J.~Tits.
\newblock {\em Groupes r\'eductifs sur un corps local},
\newblock Publications Math\'ematiques de l'IH\'ES {\bf 41} (1972), 5--251.

\bibitem[BrT84]{BruhatTits2}
F.~Bruhat and J.~Tits.
\newblock {\em Groupes r\'eductifs sur un corps local. {II}. {S}ch\'emas en groupes. {E}xistence d'une donn\'ee radicielle valu\'ee},
\newblock Publications Math\'ematiques de l'IH\'ES {\bf 60} (1984), 197--376.

\bibitem[Cap09]{Caprace-amenable}
P.-E. Caprace,
\newblock {\em Amenable groups and {H}adamard spaces with a totally disconnected isometry group},
\newblock Commentarii Mathematici Helvetici {\bf 84} (2009), no. 2, 437--455.

\bibitem[CM13]{CapraceMarquis}
P.-E.~Caprace and T.~Marquis,
\newblock {\em Open subgroups of locally compact {K}ac-{M}oody groups},
\newblock Math. Z. {\bf 274} (2013), no. 1-2, 291--313.

\bibitem[CGP15]{CGP}
B.~Conrad, O.~Gabber, and G.~Prasad,
\newblock {\em Pseudo-Reductive Groups},
\newblock Vol.~26 of {\em New Mathematical Monographs},
\newblock Cambridge University Press, Cambridge, second edition, 2015.

\bibitem[CP16]{ConradPrasad}
B.~Conrad and G.~Prasad.
\newblock {\em Classification of Pseudo-Reductive Groups},
\newblock Princeton University Press, 2016.

\bibitem[Con12]{Conrad-cosetfinite}
B.~Conrad,
\newblock {\em Finiteness theorems for algebraic groups over function fields},
\newblock Compositio Mathematica {\bf 148} (2012), no. 2, 555--639.

\bibitem[CH16]{CornulierHarpe}
Y.~Cornulier and P.~de la Harpe,
\newblock {\em Metric Geometry of Locally Compact Groups},
\newblock Vol.~25 of {\em EMS Tracts in Mathematics},
\newblock European Mathematical Society (EMS), Z\"urich, 2016.

\bibitem[SGA3]{SGA3}
M.~Demazure and A.~Grothendieck,
\newblock {\em Sch{\'e}mas en Groupes. S{\'e}minaire de G{\'e}om{\'e}trie
  Alg{\'e}brique du Bois Marie 1962-64 (SGA 3)},
\newblock Augmented and corrected 2008–2011 re-edition of the original by
  Philippe Gille and Patrick Polo. Available at
  \url{http://www.math.jussieu.fr/~polo/SGA3}.

\bibitem[DDMS99]{DixonDuSautoyMannSegal}
J.~D. Dixon, M.~P.~F. du~Sautoy, A.~Mann, and D.~Segal,
\newblock {\em Analytic pro-{$p$} Groups},
\newblock Vol.~61 of {\em Cambridge Studies in Advanced Mathematics},
\newblock Cambridge University Press, Cambridge, second edition, 1999.

\bibitem[DG70]{DemazureGabriel}
M.~Demazure and P.~Gabriel,
\newblock {\em Groupes Alg\'ebriques. {T}ome {I}: {G}\'eom\'etrie Alg\'ebrique, G\'en\'eralit\'es, Groupes Commutatifs},
\newblock Masson \& Cie, \'Editeur Amsterdam North-Holland Publishing Company, 1970.

\bibitem[Gar97]{Garrett}
P.~Garrett,
\newblock {\em Buildings and Classical Groups},
\newblock Chapman \& Hall, London, 1997.

\bibitem[GGMB14]{GabberGilleMoretBailly}
O.~Gabber, P.~Gille, and L.~Moret-Bailly,
\newblock {\em Fibr\'es principaux sur les corps valu\'es hens\'eliens},
\newblock Algebraic Geometry {\bf 1} (2014), no. 5, 573--612.

\bibitem[Lan96]{Landvogt}
E.~Landvogt,
\newblock {\em A Compactification of the {B}ruhat-{T}its Building},
\newblock Vol.~1619 of {\em Lecture Notes in Mathematics},
\newblock Springer-Verlag, Berlin, 1996.

\bibitem[Loi]{Loisel-GenerationFrattini}
B.~Loisel,
\newblock {\em Explicit generators of some pro-{$p$} groups via {B}ruhat-{T}its theory},
\newblock in preparation.

\bibitem[Mar91]{Margulis}
G.~A. Margulis,
\newblock {\em Discrete Subgroups of Semisimple {L}ie Groups},
\newblock Vol.~17 of {\em Ergebnisse der Mathematik und ihrer Grenzgebiete (3)},
\newblock Springer-Verlag, Berlin, 1991.


\bibitem[Oes84]{Oesterle}
J.~Oesterl{\'e}.
\newblock {\em Nombres de {T}amagawa et groupes unipotents en caract\'eristique {$p$}},
\newblock Inventiones mathematicae {\bf 78} (1984), no. 1, 13--88.

\bibitem[PlR94]{PlatonovRapinchuk}
V.~Platonov and A.~Rapinchuk,
\newblock {\em Algebraic Groups and Number Theory},
\newblock Vol.~139 of {\em Pure and Applied Mathematics},
\newblock Academic Press, Inc., Boston, MA, 1994.

\bibitem[PrT82]{Prasad-TitsTheorem}
G.~Prasad,
\newblock {\em Elementary proof of a theorem of {B}ruhat-{T}its-{R}ousseau and of a
  theorem of {T}its}
\newblock Bull. Soc. Math. France {\bf 110} (1982), no. 2, 197--202.

\bibitem[PrR84]{PrasadRaghunathan1}
G.~Prasad and M.~S. Raghunathan,
\newblock {\em Topological central extensions of semisimple groups over local fields},
\newblock Ann. of Math. (2) {\bf 119} (1984), no. 1, 143--201.

\bibitem[PrR85]{PrasadRaghunathan-KneserTits}
G.~Prasad and M.~S. Raghunathan,
\newblock {\em On the Kneser-Tits problem},
\newblock Commentarii mathematici Helvetici {\bf 60} (1985), 107--121.


\bibitem[RTW10]{RemyThuillierWerner}
B.~R{\'e}my, A.~Thuillier, and A.~Werner,
\newblock {\em Bruhat-{T}its theory from {B}erkovich's point of view. {I}.
  {R}ealizations and compactifications of buildings.}
\newblock Ann. Sci. \'Ec. Norm. Sup\'er. (4) {\bf 43} (2010), no. 3, 461--554.

\bibitem[RZ10]{RibesZalesskii}
L.~Ribes and P.~Zalesskii,
\newblock {\em Profinite Groups},
\newblock Vol.~40 of {\em Ergebnisse der Mathematik und ihrer Grenzgebiete. 3. Folge. A Series of Modern Surveys in Mathematics},
\newblock Springer-Verlag, Berlin, second edition, 2010.

\bibitem[Rou77]{RousseauHdR}
G.~Rousseau,
\newblock {\em Immeubles des Groupes R\'eductifs sur les Corps Locaux},
\newblock U.E.R. Math\'ematique, Universit\'e Paris XI, Orsay, 1977.

\bibitem[Ser94]{SerreCohomologieGaloisienne}
J.-P. Serre,
\newblock {\em Cohomologie Galoisienne},
\newblock Vol.~5 of {\em Lecture Notes in Mathematics},
\newblock Springer-Verlag, Berlin, fifth edition, 1994.

\bibitem[Spr98]{Springer}
T.~A. Springer,
\newblock {\em Linear Algebraic Groups},
\newblock Vol.~9 of {\em Progress in Mathematics},
\newblock Birkh\"auser Boston, Inc., Boston, MA, second edition, 1998.

\bibitem[Tit64]{Tits-AlgebraicAbstract}
J.~Tits,
\newblock {\em Algebraic and abstract simple groups},
\newblock Annals of Mathematics {\bf 80} (1964), no. 2, 313--329.

\bibitem[Tit79]{TitsCorvallis}
J.~Tits,
\newblock {\em Reductive groups over local fields},
\newblock In {\em Automorphic Forms, Representations and {$L$}-Functions,
{P}roc. {S}ympos. {P}ure {M}ath., {O}regon {S}tate {U}niv., {C}orvallis,
  {O}re., 1977},
\newblock {P}art 1, Proc. Sympos. Pure Math., XXXIII, 29--69,
\newblock Amer. Math. Soc., Providence, R.I., 1979.


\end{thebibliography}


\end{document}